\documentclass[a4paper,12pt]{article}
\usepackage{amsmath,amsfonts,amsthm,amscd,amssymb,latexsym,comment}
\usepackage{enumitem}
\usepackage{tikz}  %
\usetikzlibrary{calc,trees,positioning,arrows,fit,shapes,calc,backgrounds}
\usepackage{graphicx,psfrag}
\usepackage[mathscr]{eucal}
\usepackage{lipsum}
\usepackage[hang]{footmisc}
\includecomment{vnr}
\title{Automorphisms of Partially Commutative Groups III: Inversions and Transvections 
}
\author{Andrew J. Duncan,
 Vladimir N. Remeslennikov}

\def\nul{\emptyset }

\def\d{\delta }
\def\b{\beta }

\def\i{\iota }

\def\e{\varepsilon }
\def\G{\Gamma }
\def\g{\gamma }
\def\a{\alpha }
\def\t{\tau }

\def\s{\sigma }

\def\W{\Omega}

\def\cK{{\cal{K}}}
\def\cL{{\cal{L}}}
\def\cM{{\cal{M}}}

\def\cS{{\cal{S}}}

\newtheorem{theorem}{Theorem}[section]
\newtheorem{lemma}[theorem]{Lemma}
\newtheorem{corol}[theorem]{Corollary}
\newtheorem{prop}[theorem]{Proposition}

\newtheorem{definition}[theorem]{Definition}
\newtheorem*{defn*}{Definition}

\newtheorem{exam}[theorem]{Example}
\newenvironment{example}{\begin{exam} \rm}{\end{exam}}
\newtheorem{remk}[theorem]{Remark}
\newenvironment{remark}{\begin{remk} \rm}{\end{remk}}

\newtheoremstyle{citing}
{6pt}{6pt}
{\itshape}
{0\parindent}{\bfseries}
{.}
{ }
{\thmnote{#3}}
\theoremstyle{citing}
\numberwithin{equation}{section}
\numberwithin{figure}{section}

\renewcommand{\v}{{\operatorname{v}}}

\newcommand{\Aut}{\operatorname{Aut}}

\newcommand{\out}{{\operatorname{{out}}}}
\newcommand{\innt}{{\operatorname{{s}}}}
\newcommand{\supp}{{\operatorname{{Supp}}}}

\newcommand{\FF}{\ensuremath{\mathbb{F}}}
\newcommand{\GG}{\ensuremath{\mathbb{G}}}
\newcommand{\ZZ}{\ensuremath{\mathbb{Z}}}

\newcommand{\NN}{\ensuremath{\mathbb{N}}}

\newcommand{\la}{\langle}
\newcommand{\ra}{\rangle}

\newcommand{\cl}{\operatorname{cl}}

\newcommand{\maps}{\rightarrow}

\newcommand{\bs}{\backslash}
\newcommand{\St}{\operatorname{St}}
\newcommand{\conj}{\operatorname{conj}}
\newcommand{\cSt}{\St^{\conj}}

\newcommand{\Conj}{\operatorname{Conj}}

\newcommand{\LInn}{\operatorname{LInn}}

\newcommand{\Inn}{\operatorname{Inn}}

\newcommand{\Tr}{\operatorname{Tr}}

\newcommand{\tr}{\tau}

\newcommand{\Inv}{\operatorname{Inv}}

\newcommand{\cmp}{{\operatorname{{\cal{C}}}}}

\newcommand{\GL}{\operatorname{GL}}

\newcommand{\lk}{\operatorname{{lk}}}
\newcommand{\st}{\operatorname{{st}}}
\newcommand{\stab}{\operatorname{{stab}}}

\newcommand{\ad}{\mathfrak{a}}
\newcommand{\Ad}{\mathfrak{A}}
\newcommand{\V}{{\operatorname{v}}}

\newcommand{\be}{\begin{enumerate}}
  \newcommand{\ee}{\end{enumerate}}
\newcommand{\biz}{\begin{itemize}}
\newcommand{\eiz}{\end{itemize}}
\newcommand{\bd}{\begin{description}}
\newcommand{\ed}{\end{description}}

\newlength{\nts}
\newlength{\rts}
\newlength{\lts}
\newenvironment{ajd1}{\noindent\color{blue} ~\\AJD }{}

\newcommand\blfootnote[1]{%
  \begingroup
  \renewcommand\thefootnote{}\footnote{#1}%
  \addtocounter{footnote}{-1}%
  \endgroup
}
\begin{document}

\maketitle

\begin{abstract}
The structure of a certain  subgroup $\St$ of the automorphism group of a partially commutative group (RAAG) 
$\GG$ is described in detail: namely the subgroup generated by inversions and elementary transvections. 
We define admissible subsets of the generators of $\GG$, and show that  $\St$ is the subgroup of 
automorphisms which fix all subgroups $\la Y\ra$ of $\GG$, for all admissible subsets $Y$.  
A decomposition of $\St$ as an iterated tower of semi-direct products in given and the structure
of the factors of this decomposition described. The construction allows a presentation of $\St$ to be computed,
from the commutation graph of $\GG$.  
\end{abstract}
\blfootnote{\noindent\emph{ Mathematics Subject Classification.} 
Primary  20F36, 20F28; Secondary  20F05.}

\section{Introduction}
A {\em partially commutative group} (also known 
as a {\em right-angled Artin group})
is a group given by a finite presentation
$\la X| R\ra$, where $R$ is a subset of
$\{[x,y]\,|\, x,y\in X, x\neq y\}$. (Our convention
is that $[x,y]=x^{-1}y^{-1}xy$.)  
The \emph{commutation graph} of a partially commutative group is the
simple graph $\G$ with vertices $X$ and an edge joining $x$ to
$y$ if and only if $[x,y]\in R$. (A simple graph
is one without multiple edges or self-incident vertices.)
A simple graph $\G$ uniquely determines a presentation $\la X|R\ra$ of a
partially commutative group with commutation graph $\G$, which we denote 
$\GG_\G$, and if $\G$ and $\G'$ are  simple graphs such that $\GG_\G\cong \GG_{\G'}$ then
$\G$ and $\G'$ are isomorphic graphs \cite{droms87}. The study of isomorphisms between
partially commutative groups therefore reduces to the study of  automorphisms of groups  
$\GG_\G$.  For background information on automorphisms of partially commutative
groups we refer to \cite{CharneyStambaughVogtmann}, \cite{Day09}, \cite{DR6}  and the references therein.
In particular, the automorphism group $\Aut(\GG_\G)$ of $\GG_\G$ was shown to have a finite generating
set by Laurence, building on work of Servatius \cite{Laurence95,servatius89}; a finite presentation
for these groups was found by Day \cite{Day09}; and geometric models of the Outer automorphism group
of $\G_\G$ were constructed in \cite{CharneyStambaughVogtmann}. Here we consider the
decomposition of $\Aut(\GG_\G)$ into subgroups corresponding to particular types
of the generators found by Laurence and Servatius.

Laurence and Servatius identified four types of elementary automorphism which together generate
$\Aut(\GG_\G)$. These are
\begin{itemize}
\item automorphisms which permute the elements of $X$, called \emph{graph automorphisms},   
\item automorphisms which map an element $x\in X$ to $x^{-1}$ and fix all other elements of $X$, 
called \emph{inversions},
\item 
  \emph{elementary transvections} which map an    element $x\in X^{\pm 1}$ to $xy^{\pm 1}$, for some element $y\in X$ and fix all elements of
  $X\bs \{x\}$, and 
\item \emph{vertex conjugating automorphisms} which, for some element $x\in X^{\pm 1}$ and some subset $C\subseteq X$, map
  $c$ to $c^x$ and fix all elements of $X\bs C$. 
\end{itemize}
Conditions on elements and subsets of $X$ under which elementary transvections and vertex conjugating automorphisms exist
are discussed in Section \ref{sec:prelim} below.  
The subgroup $\Aut^*(\GG_\G)$ generated by inversions, elementary transvections and elementary vertex conjugating automorphisms
has finite index; and $\Aut(\GG_\G)=\Aut(\G^\cmp)\ltimes \Aut^*(\GG_\G)$, where
$\Aut(\G^\cmp)$ is a subgroup of the group of automorphisms of $\GG_\G$ which permute $X^{\pm 1}$ 
(see Section \ref{sec:genaut} below for more detail). 

Generalising the notion of vertex conjugating automorphism: an automorphism $\phi\in \Aut(\GG)$ is called  a {\em conjugating} 
automorphism if there exists $g_x\in \GG$ such that $x\phi=x^{g_x}$, for all $x\in X$. 
The subgroup of $\Aut(\GG)$ consisting of   all  conjugating automorphisms 
 is denoted $\Conj(\GG)$. Laurence \cite{Laurence95} proved that  $\Conj(\GG)$ is the group generated by the vertex conjugating
automorphisms and later
Toinet \cite{toinet12} constructed  a finite presentation for this  group (with generators the vertex conjugating automorphisms).   
 Here we give a description of the
 structure of the subgroup generated by inversions and elementary transvections. We use the methods of \cite{DR6}, where a
 characterisation of $\Aut^*(\GG_\G)$ was given, in terms of stabilisers; which we shall now describe.

 For $x\in X$, the \emph{link}, $\lk(x)$, of $x$ is the set of all vertices joined to $x$ by an edge of $\G$. The \emph{star}, $\st(x)$, of $x$ is
 $\lk(x)\cup \{x\}$.
 We define an equivalence relation $\sim$ on $X$ by $x\sim y$ if and only if either $\st(x)=\st(y)$ or $\lk(x)=\lk(y)$; and denote by $[x]$ the
 $\sim$ equivalence class of $x$. (See Section \ref{sec:prelim} for more detail.)
 The \emph{admissible set}, $\ad(x)$ of $x$ is 
 \[\ad(x)=\cap_{y\in \lk(x)} \st(y),\]
 and we define \[\cK=\{\ad(x)\,|\,x\in X\},\] 
the set  of all admissible sets.  
  (See Example \ref{ex:admiss} below.)
 
  For any subset $Y$ of $X$ we denote by $\GG(Y)$ the subgroup of $\GG$ generated by $Y$.
In \cite{DR6} we defined  
\[
\St(\cK)=\{\phi\in \Aut(G)|G(Y)\phi=G(Y), \textrm{ for all } 
Y\in \cK\}
\]
and 
\[
\cSt(\cK)=\{\phi\in \Aut(G)|G(Y)^\phi=G(Y)^{f_Y}, 
\textrm{ for some } f_Y\in G, \textrm{ for all } 
Y\in \cK\},
\]
and proved that $\Aut^*(\GG_\G)=\cSt(\cK)$. As inversions and elementary transvections all belong to $\St(\cK)$, 
it follows that $\Aut^*(\GG)$ is generated by
$\St(\cK)$ and $\Conj(\GG_\G)$. However, in general these two subgroups intersect
non-trivially and it is not the case that $\Aut^*(\GG_\G)=\St(\cK)\cdot \Conj(\GG_\G)$.  Necessary and sufficient conditions 
on the graph $\G$ 
under which the latter holds are given in \cite{DR6}. 

In this paper we give a decomposition of $\St(\cK)$ as chain of semi-direct products, of subgroups
whose structure we can, to a significant extent, understand.
 To this end, the \emph{height} $h(x)$ of an element $x\in X$ is defined to be
the largest integer $i$ such that there exists a strictly descending
chain $\ad(x_i)>\ad(x_{i-1})>\cdots>\ad(x_0)$, where $x_i=x$. The 
$\cK$\emph{-height} of $h_\cK(\GG)$ of $\GG$ is the maximum  of the heights 
of elements of $X$. 
Let $h_\cK=h_\cK(\GG)$ and 
for  $0\le k\le h_\cK$, let the \emph{level $k$ vertex set} of $X$ be
\[\V(k)=\{y\in X| h(y)=k\}.\]
(See Example \ref{ex:admiss}.) We define
\[St_{k}^{\V}(\cK)=\{\phi\in \St(\cK)| y\phi=y, \textrm{ for all } y\in X\bs \V(k)\}\]
and 
\[St_{x}^{\V}(\cK)=\{\phi\in \St(\cK)| y\phi=y, \textrm{ for all } y\in X\bs [x]\}.\]
Our main results are the following, where we write $\St_k^\v$ and $\St^\v_x$ for  $\St_k^\v(\cK)$ and $\St^\v_x(\cK)$, respectively.
\begin{theorem}\label{thm:st2decomp}
Let  $\cmp$ be a transversal for $\sim$ and let $\cmp(k)=\v(k)\cap \cmp$. Then
\be[label=(\roman*)]
\item \label{it:st2decomp2} 
$
\St(\cK)=
(\cdots (
\St_{0}^\v
\ltimes
\St_{1}^\v
)
\ltimes 
\cdots
\ltimes
\St_{h_\cK-1}^\v
)
\ltimes
\St_{h_\cK}^\v
$ 
and 
\item\label{it:st2decomp4} $\St_{k}^\v=\prod_{y\in \cmp(k)}\St_{y}^\v$,
for $k=0,\ldots ,h_\cK$. 
\ee
\end{theorem}
(See Example \ref{ex:admiss0}.)
This leaves the  structure of $\St_{x}^\v$ to be determined. There are two cases
to consider, which depend on the size of a further set, the \emph{closure} of an element $x$ of $X$, defined
as
\[\cl(x)=\cap_{y\in \st(x)} \st(y).\]
(See Section \ref{sec:prelim} for details.) 
As 
$\cl(x)=\ad(x)\cap \st(x)$ we always have $\cl(x)\subseteq \ad(x)$.
If $\cl(x)=\ad(x)$, then   $\GG([x])$ is a free Abelian group and consequently $\St^\v_x$ has the following form, 
where, for positive integers $a,b$, we denote the group of $a\times b$ integer matrices under addition by  $\cM(a,b)$.
\begin{theorem}[\textit{cf.} Theorem \ref{thm:stxva}]\label{thm:stxva0}
  Let $x\in X$ such that $\ad(x)=\cl(x)$. Assume  and  that
  $|\ad(x)|=r$ and $|[x]|=s$. Then $s\le r$ and  
\[\St_x^{\v}(\cK)\cong \GL(s,\ZZ)\ltimes_\theta \cM(s,r-s),\]
where,  
  for $A\in \GL(s,\ZZ)$ and $B\in \cM(s,r-s)$, the automorphism $A\theta$ maps $B$ to $A^{-1}B\in  \cM(s,r-s)$.
\end{theorem}
On the other hand, 
if  $\cl(x)$ is a proper subset of $\ad(x)$ 
then  $\GG([x])$ is a free group. In this case  we define $\ad_\out(x)=\ad(x)\bs \cl(x)$
and we have the following decomposition of $\St^\v_x$. 
\begin{theorem}[\textit{cf.} Theorem \ref{thm:stxdecomp}]\label{thm:stxdecomp0}  
Let $x\in X$ such that $\ad(x)\neq \cl(x)$. Assume  and  that
  $|\cl(x)|=q$ and $|[x]|=p$. Then $p\le q$ and $\St^\v_x$ has subgroups $\St_{x,l}^\v$ and $\St_{x,s}^\v$ such that 
\[\St_x^\v= \St_{x,l}^\v \ltimes \St_{x,s}^\v,\] 
\[\St_{x,l}^\v=\{\phi\in \St_x^\v\,|\, y\phi \in \GG([x]\cup \ad_\out(x)), \forall y\in [x]\}\] and
\[\St_{x,s}^\v\cong \cM(p,q-p).\]
\end{theorem}
 Although we do not have a structural decomposition we 
give a finite presentation  of $\St_{x,l}^\v$ 
in Theorem \ref{thm:stxvf}. Combining these theorems allows us to find generators of $\St(\cK)$, giving our final result.
\begin{corol}\label{cor:stgens} 
  The subgroup  of $\Aut(\GG_\G)$ generated by the set of all  inversions and elementary transvections is precisely 
  $\St(\cK)$. Moreover, $\St(\cK)$ has a finite presentation with
  these generators. 
\end{corol}
Indeed, such a finite presentation of $\St(\cK)$ may be explicitly constructed
from the decomposition  appearing in the theorems above. 
(See Examples \ref{ex:admiss0}, \ref{ex:admiss1}, \ref{ex:admiss2} and  \ref{ex:admiss3}.)

In Section \ref{sec:prelim} we cover the necessary background on partially commutative groups, admissible sets and closure, and generators 
of the automorphism group of $\GG_\G$. Section \ref{sec:stk} contains the proof of Theorem \ref{thm:st2decomp}. Section \ref{sec:stxv} is
concerned with $\St^\v_x$. In Section \ref{sec:stxvc} a more detailed version of Theorem \ref{thm:stxva0} is stated and proved,
namely, Theorem \ref{thm:stxva0}. In Section \ref{sec:stxva} we prove Theorem \ref{thm:stxdecomp}, which is  a more 
detailed version of Theorem \ref{thm:stxdecomp0} and then define generators and relations for $\St_{x,l}^\v$. The remainder of the 
paper consists of the proof of Theorem \ref{thm:stxvf}. For this we use peak reduction, constructing a  modification of  
the process of \cite{Day09} to 
work within the given generating set of $\St_{x,l}^\v$. 
\section{Preliminaries}\label{sec:prelim}
Throughout this article, let $\GG=\GG_\G$ be the partially commutative group with commutation graph $\G$ and presentation $\la X|R\ra$, as above. For 
$Y\subset X$ 
the subgroup $\GG(Y)$  of $\GG$ generated by $Y$ is also a partially commutative group with commutation graph equal to 
the full subgraph of $\G$ induced by $Y$ \cite{Baudisch77}. 

For $w\in \GG$ denote by $\supp(w)$  the minimal subset $Y$ of $X$ such that 
$w\in \GG(Y)$. 
The \emph{length} $|w|$ of an element $w$ of $\GG$ is the minimum of the lengths of words in $F(X)$ representing
$w\in\GG$. If $u$ is a word of $F(X)$ of $F(X)$-length equal to the length $|u|$ of $u$ in $\GG$, then we say $u$ is a \emph{minimal} word. If $u$ and $v$ are
minimal words such that $|uv|=|u|+|v|$, we write $uv=u\circ v$. 

We extend definitions of star and link from single elements of $X$ to subsets of $X$: for $Y\subset X$  define the \emph{star} of $Y$ to be 
$\st(Y)=\cap_{x\in Y}\st(y)$. 
By convention we set $\st(\nul) = X$. 
We define the {\em closure} of $Y$ to be 
$\cl(Y)=\st(\st(Y))$.
The closure operator on $\G$ satisfies, among other things, the properties that $\cl(Y)$ is a simplex
(i.e.  the full subgraph on $\cl(Y)$ 
is  a
complete graph) and for $x\in X$, the closure $\cl(x)$ is the maximal simplex 
contained in $\st(x)$. 
  We set \[\cL=\{\cl(x)\,|\, x\in X\}.\] 
(See \cite[Lemma 2.4]{DKR3} for further details.)

As in \cite[Lemma 2.5]{DR6}, we have 
$\ad(x)=\cl(x)$ if and only if $\ad(x)\subseteq \st(x)$; from which
it follows that   $\ad(x)=\cl(x)$ if and only if $\ad(x)$ is a simplex.
%
The following straightforward lemmas are proved in \cite{DR6}.
\begin{lemma}\label{lem:ad0}
For all $x,y,z \in X$, the following hold.
\be[label=(\roman*)]
\item\label{it:ad8} If $y\in \ad(x)$ then $\ad(y)\subseteq \ad(x)$.
\item\label{it:ad11} If $[x,y]=1$ then $[\GG(\ad(x)),\GG(\ad(y))]=1$.
\item\label{it:ad12}  $\ad(y)\subseteq \ad(x)$ if and only if
$\lk(x) \subset \st(y)$.
\item\label{it:adcl1} $\ad(x)=\ad(z)$ if and only if  $z\in [x]$.
\item\label{it:adcl2}
$[x]=
\ad(x)\bs(\cup\{\ad(y)|y\in \ad(x) \textrm{ and }\ad(y) \subsetneq \ad(x)\})$.
\ee
\end{lemma}

Let $\sim_{\st}$ be the relation on $X$ given by $x\sim_{\st} y$ if
and only if $\st(x)=\st(y)$ and  $\sim_{\lk}$ be the relation given by
$x\sim_{\lk} y$ if
and only if 
$\lk(x)=\lk(y)$. 
 These are equivalence relations
and the equivalence classes of $x$ under $\sim_{\st}$ and 
$\sim_{\lk}$  are denoted by $[x]^{\st}$ and $[x]^{\lk}$, respectively.
Moreover $\sim=\sim_{\st}\cup \sim_{\lk}$. 
 In addition (see \cite[Lemma 2.7]{DR6} for details) 
if $\ad(x)=\cl(x)$ then $[x]=[x]^{\st}$ and otherwise $[x]=[x]^{\lk}$. 

 Let $L=X\cup X^{-1}$ and for $x\in L$ let $\v(x)=X\cap \{x,x^{-1}\}$. 
We extend the notation for stars, links, closures and admissible sets from $X$ to $L$ as follows.
\biz
\item
For $x\in L$, let $\st(x)$, $\lk(x)$, $[x]$, $\ad(x)$ and $\cl(x)$ denote 
$\st(\v(x))$, $\lk(\v(x))$, $[\v(x)]$, $\ad(\v(x))$ and $\cl(\v(x))$, respectively, and similarly for $[x]^{\st}$ and $[x]^{\lk}$. 
\item
  For $o$ equal to any one of the operators $\st$, $\lk$, $[~]$, $\ad$ or $\cl$ above, 
let $o_L(x)$ denote $o(x)\cup o(x)^{-1}$; so  $st_L(x)=\st(x)\cup \st(x)^{-1}$, etc..
\eiz
\subsection{Generators for $\Aut(\GG)$}\label{sec:genaut}

First we describe the  conditions  under which elementary transvections and vertex conjugating automorphisms exist, 
then we define the subgroup $\Aut(\G^\cmp)$
and finally 
we extend the definitions of Laurence and Servatius to give a larger generating set, which is 
convenient for peak reduction proofs.

For $x,y\in L$, with $x\neq y^{\pm 1}$ there exists an elementary transvection in $\Aut(\GG)$ mapping $x$ to $xy$ if and only if 
$\lk(x)\subseteq \st(y)$ (see for example \cite{servatius89}). 
Given $y\in L$ and $T\subset L\bs\{ y^{\pm 1}\}$ such that $T\cap T^{-1}=\nul$ and 
$\lk(t)\subseteq \st(y)$, for all $t\in T$;  
the automorphism $\tr_{L,y}=\prod_{t\in T}\tr_{t,y}$ is called 
a \emph{transvection}. 
 
Let $x$ in $L$ and $C\subseteq X\bs \st(x)$. 
Then  there exists an automorphism of $\GG$ 
mapping $c\in C$ to $x^{-1}cx$, and fixing all other elements of $X$, if and only if  $C$ is the vertex set of a union of connected components of 
$\G\bs \st(x)$; the graph obtained from $\G$ by removing all vertices of $\st(x)$ and all their incident edges.
(see for example \cite{servatius89}). 
We denote this vertex conjugating automorphism by $\a_{C,x}$. 
If $C$ 
consists of the vertices of  a single connected component of $\G\bs \st(x)$ then $\a_{C,x}$ is called an \emph{elementary} 
vertex conjugating automorphism.  

For ease of reference we make the following definitions. 
\begin{definition}
Denote by 
\be
\item  $\Aut(\G^\pm)$ the  subgroup of automorphisms which permute $L$;
\item
$\Inv=\Inv(\GG)$ the set of inversions;
\item 
$\Tr=\Tr(\GG)$ the set of elementary transvections;
\item $\LInn=\LInn(\GG)$ the set of elementary vertex conjugating automorphisms. 
\ee
The set of all transvections is denoted $\Tr^+$ and the set of all vertex conjugating automorphisms by $\LInn^+$. 
\end{definition}

The group  $\Aut^\cmp$, mentioned in the introduction depends on the choice of 
an ordering on each of the sets $[x]$.  Choose  a total order $<$ on  each set $[x]\subseteq X$. 
Then $\Aut^\cmp$ is defined to be the group of automorphisms
of $\GG_\G$ which permute $L$ and respect the order on $[x]$, for all $x\in X$. That is, an automorphism $\phi$ belongs to $\Aut^\cmp$ if 
it belongs to $\Aut(\G^\pm)$, and whenever  $u,v\in [x]$, with $u<v$, then $u\phi <v\phi$. (See \cite{DR6} for details.) As 
$\Aut(\GG)=\Aut(\G^\pm)\ltimes\Aut^*(\GG)$ the focus of attention is the subgroup $\Aut^*(\GG)=\la \Inv,\Tr,\LInn\ra$.   

 In the sequel we shall make use of a larger set  than the Laurence-Servatius generators, known as  Whitehead automorphisms, 
to generate $\Aut(\GG)$. These  originate in work of Whitehead, and were developed by 
Rapaport, Higgins and Lyndon, and McCool, to study  Automorphisms of Free groups, using peak reduction. Day \cite{Day09} defined Whitehead 
automorphisms over partially commutative groups and used  them in peak reduction arguments, to construct finite presentations of 
their automorphism groups.  
\begin{definition}\label{def:wh}
A \emph{Whitehead automorphism} is an element of $\Aut(\GG)$ of one of two types.
\begin{description}
\item[Type 1.] Elements of  $\Aut(\G^{\pm})$.
\item[Type 2.] Elements of the form $\a_{C,x}\tr_{T,x}$, where $\a_{C,x}\in \LInn^+$, 
$\tr_{T,x}\in \Tr^+$, and $(C\cup C^{-1})\cap T=\nul$. 
\end{description}
In the definition of Type 2 elements we allow $\a_{C,x}$ or $\tr_{T,x}$, but not both, to 
be trivial; so $\LInn^+$ and $\Tr^+$ are sets of Whitehead automorphisms of Type 2. 
\begin{description}
\item[Notation] 
The Whitehead automorphism $\a_{C,x}\tr_{T,x}$ of Type 2 is denoted $(A,x)$, where $A$ is any subset of $L$ such that 
\be
\item $x\in A$ and $x^{-1}\notin A$;
\item $A\bs \{x\}$ is the disjoint union of the set $C\cup C^{-1}$, the set $T$ and a set $U\cup U^{-1}$, where $U$ is  some subset $\lk(x)\subset X$, 
such that $(U\cup U^{-1})\cap T=\nul$. (We always assume $C\cup T$ is not empty.)
\ee
\end{description}

The set of all Whitehead automorphisms is denoted $\W$. 
\end{definition}
\begin{remark}\label{rem:whrem}
\be
\item \label{it:whrem1}
The notation $(A,x)$ uniquely determines an automorphism $\phi$ say, although there may be more than one expression 
of $\phi$ in terms of Laurence-Servatius generators. 
For example, if $\lk(x)\subseteq \st(y)$ then $(\{x,x^{-1},y\},y)=\tr_{x,y}\tr_{x^{-1},y}=\a_{\{x,x^{-1}\},y}$. 
\item\label{it:whrem2} For $x\in L$ and $A\subset L$ the pair $(A,x)$ denotes a Whitehead automorphism if and only if 
$A=(C\cup U)^{\pm 1}\cup T\cup\{x\}$, for some  $C\cup U\subset X$ and $T\subset L$ such that $T\cap (C\cup U)^{\pm 1}=\nul$, $T\cap T^{-1}=\nul$,  
$C$ is a union of connected components of $\G\bs \st(x)$, $U\subseteq \lk(x)$, and $T\subseteq \ad_L(x)\bs \{x^{\pm 1}\}$. In this 
case $(A,x)=\a_{C,x}\tr_{T,x}$. 
\ee
\end{remark}

Day \cite{Day09} defines a Whitehead automorphism $\phi$ to be 
\be[label=(\roman*),ref=(\roman*)]
\item
\emph{long range} if either $\phi$ is of Type 1; or $\phi$ is of Type 2, $\phi=(A,a)$ and $y\phi=y$, for all $y\in \st(a)$, and
\item
\emph{short range} if it is of 
Type 2,  $\phi=(A,a)$  and $y\phi =y$, for all $y\in X\bs \st(a)$.
\ee
\begin{remark}\label{rem:lsfac}
In general, if  $\phi=(A,a)$ is of Type 2,  and we set  $A_s=A\cap \st(a)_L$ and $A_l=A\bs A_s$ then 
$\phi_s=(A_s,a)$ is short range, $\phi_l=(A_l\cup \{a\},a)$ is long range and $\phi=\phi_s\phi_l$.  Hence every
Whitehead automorphism factors uniquely as  a product of a short range and a long range automorphism.
\end{remark}
\begin{definition}
The set of short range automorphisms is denoted $\W_s$ and the set
of long range automorphisms is denoted $\W_l$.  
\end{definition}
As the Laurence-Servatius generators are all either short or long range 
Whitehead automorphisms it follows that $\Aut(\GG)$ is generated by the union
$\W_s\cup \W_l$ of  short and long range Whitehead automorphisms. 

 Day \cite{Day09} shows that $\Aut(\GG)$ has a finite presentation with generators $\W_s\cup \W_l$ and a set of relations R,  
partitioned into subsets R1--R7, which we shall refer to as DR1--DR7 in the sequel. 
\section{The structure of $\St(\cK)$}\label{sec:stk}

The decomposition of $\St(\cK)$ reflects the structure of  the partial order, by inclusion, on the  set $\cK$ 
which we stratify as 
follows.
%
Let
the \emph{level $k$ admissible set} of $X$ be 
\[\Ad(k)=\bigcup_{i=0}^k\bigcup_{y\in \V(i)} \ad(y).\] 

With this notation $\Ad(h_\cK)=X$, $\V(h_\cK)=\{y\in X| h(y)=h_\cK\}$,   
 \[\Ad(0)=\V(0)=\{y\in X|  
\ad(y)=[y]\},\]   and it follows from Lemma \ref{lem:ad0} \ref{it:adcl2}  
that 
\begin{align*}
\Ad(k)&=\bigcup_{y\in \V(k)} \left[[y] \cup \bigcup_{z\in \V(k-1)}\ad(z)\right]\cup \Ad(k-1)\\
&=\bigcup_{y\in \V(k)} [y] \cup \Ad(k-1)\\
&=
\V(k)\cup \Ad(k-1)
\end{align*}
and $\V(k)\cap \Ad(k-1)=\emptyset$. 
Moreover, if $\cmp$ is a transversal for $\sim$  (a set of representatives of equivalence
classes), then setting $\cmp (k)=\cmp\cap\V(k)$, we have $\V(k)=\bigcup_{y\in \cmp(k)}[y]$.
\begin{example}\label{ex:admiss}
Let $\G$ be the graph:
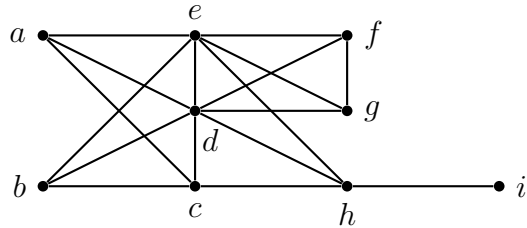
\begin{center}
\begin{tikzpicture}[ele/.style={fill=black,circle,minimum width=4pt,inner sep=1pt},every fit/.style={ellipse,draw,inner sep=-2pt}]
  \node[ele,label=left:$a$] (a1) at (0,3) {};
  \node[ele,label=left:$b$] (a2) at (0,1) {};    

  \node[ele,,label=above:$e$] (b1) at (2,3) {};
  \node[ele,,label={[shift={(.2,-0.75)}]$d$}] (b2) at (2,2) {};
  \node[ele,,label=below:$c$] (b3) at (2,1) {};
  
  \node[ele,,label=right:$f$] (c1) at (4,3) {};
  \node[ele,,label=right:$g$] (c2) at (4,2) {};
  \node[ele,,label=below:$h$] (c3) at (4,1) {};

  \node[ele,,label=right:$i$] (d3) at (6,1) {};

  \draw[thick] (a1) -- (b1);
  \draw[thick] (a1) -- (b2);
  \draw[thick] (a1) -- (b3);
 
  \draw[thick] (a2) -- (b1);
  \draw[thick] (a2) -- (b2);
  \draw[thick] (a2) -- (b3);

  \draw[thick] (b1) -- (b2);
  \draw[thick] (b1) -- (c1);
  \draw[thick] (b1) -- (c2);
  \draw[thick] (b1) -- (c3);

  \draw[thick] (b2) -- (b3);
  \draw[thick] (b2) -- (c1);
  \draw[thick] (b2) -- (c2);
  \draw[thick] (b2) -- (c3);

  \draw[thick] (b3) -- (c3);

  \draw[thick] (c1) -- (c2);

  \draw[thick] (c3) -- (d3);
 
 \end{tikzpicture}
\end{center}

Then $[a]=\{a,b\}$, $[f]=\{f,g\}$, all other equivalence classes contain a single element, 
$\ad(a)=\ad(b)=\{a,b,d,h\}$, $\ad(c)=\{c,d,e\}$, $\ad(d)=\{d\}$, $\ad(e)=\{d,e\}$, $\ad(f)=\ad(g)=\{d,e,f,g\}$, $\ad(h)=\{h\}$ and 
$\ad(i)=\{c,d,e,h,i\}$; with inclusions as shown in the following diagram. 
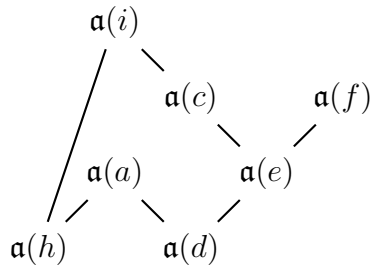
\begin{center}
\begin{tikzpicture}[ele/.style={fill=black,circle,minimum width=4pt,inner sep=1pt},every fit/.style={ellipse,draw,inner sep=-2pt}]
  \node (a4) at (0,1) {$\ad(h)$};

  \node (b1) at (1,4) {$\ad(i)$};    
  \node (b3) at (1,2) {$\ad(a)$};    

  \node (c2) at (2,3) {$\ad(c)$};    
  \node (c4) at (2,1) {$\ad(d)$};    

  \node (d3) at (3,2) {$\ad(e)$};    

  \node (e2) at (4,3) {$\ad(f)$};   
 
 \draw[thick] (a4) -- (b1);
 \draw[thick] (a4) -- (b3);

 \draw[thick] (b1) -- (c2);
 \draw[thick] (b3) -- (c4);

 \draw[thick] (c2) -- (d3);
 \draw[thick] (c4) -- (d3);

 \draw[thick] (d3) -- (e2);
 \end{tikzpicture}
\end{center}
We take a transversal $\cmp=\{a,c,d,e,f,h,i\}$ for $\sim$. 
Thus $h_\cK(\GG)=3=h(i)$, $\v(3)=\cmp(3)=\{i\}$, $\v(2)=\{c,f,g\}$, $\cmp(2)=\{c,f\}$, $\v(1)=\{a,b,e\}$, $\cmp(1)=\{a,e\}$ and $\v(0)=\cmp(0)=\{d,h\}$.   
Finally $\Ad(3)=\ad(i)\cup \ad(f)\cup \ad(a)=X$, $\Ad(2)=\ad(c)\cup \ad(f)\cup \ad(a)=X\bs\{i\}$, $\Ad(1)=\ad(a)\cup \ad(e)=\{a,b,d,e,h\}$ and $\Ad(0)=\ad(d)\cup \ad(h)=\{d,h\}$.  
\end{example}

Let  $k\in \NN$ such that $0\le k\le h(x)$. Define 
the \emph{level $k$ stabiliser} of $\cK$ to be
\[St_{k}(\cK)=\{\phi\in \St(\cK)| y\phi=y, \textrm{ for all } y\in X\bs\Ad(k)\},.\]
With this definition, for $x\in X$ of height $k$, we have
\[\St_{x}^{\V}(\cK)\le St_{k}^{\V}(\cK)\le St_{k}(\cK).\]
We shall abbreviate $\St(\cK)$, $\St_{k}(\cK)$, $\St_{k}^\V(\cK)$ and $\St_x^\V(\cK)$ to 
$\St$, $\St_{k}$, $\St_{k}^\V$ and $\St_x^\V$, respectively, when no ambiguity arises.

\begin{definition}\label{defn:stx2}
Let $h_\cK=h_\cK(\GG)$ and let $\phi\in \St(\cK)$.  For
each $k$ such that $0\le k\le h_\cK$ define the \emph{level $k$ restriction} 
$\phi_k$ of $\phi$ to be the map given by
\begin{equation*}
z\phi_k=
\left\{
\begin{array}{ll}
z, & \textrm{ if } z\notin \Ad(k)\\
z\phi, &\textrm{ if } z\in \Ad(k),
\end{array}
\right.
\end{equation*}
for all $z\in X$.
\end{definition}

\begin{lemma}\label{lem:st2}
For $\phi\in \St(\cK)$ and $0\le k\le h_\cK$ 
the map $\phi_k$ extends uniquely to an 
element of $\St(\cK)$ (also denoted $\phi_k$). 
Moreover 
the 
 map $s_{k}$, such that $\phi\in \St(\cK)$ is mapped to $\phi_k$ is a retraction 
of $\St(\cK)$ onto 
$\St_{k}(\cK)$.
\end{lemma}
\begin{proof}
First we show that if $\phi\in \St(\cK)$ then $\phi_k$ extends to an 
endomorphism of $\GG$, which is necessarily unique as the images of generators are determined. 
Suppose that $a,b\in X$ such that $[a,b]=1$.
It suffices to show that $[a\phi_k,b\phi_k]=1$. 
If either $\{a,b\}\subseteq \Ad(k)$ or $\{a,b\}\cap \Ad(k)=\nul$ then
clearly $[a\phi_k,b\phi_k]=1$.  This leaves the case $a\notin \Ad(k)$,
$b\in \Ad(k)$.  As $\phi \in \St(\cK)$ we have
$b\phi\in \GG(\ad(b))$ and so, from Lemma \ref{lem:ad0} \ref{it:ad11},
$[a\phi_k,b\phi_k]=[a,b\phi]=1$. Therefore $\phi_k$ is an endomorphism
of $\GG$. 

To see that $\phi_k$ is an automorphism 
suppose first that $\phi,\psi\in \St(\cK)$. If $z\notin \Ad(k)$ then $z\phi_k\psi_k=z$, while for $z\in \Ad(k)$,
$z\phi_k\psi_k=(z\phi)\psi_k$. By definition, $z\in \Ad(k)$ implies $z\in \ad(y)$, for some 
$y\in X$, with $h(y)\le k$. Since $\phi\in St(\cK)$ we have $z\phi\in \GG(\ad(y))$, so $z\phi\in \GG(\Ad(k))$ and 
 $\supp(z\phi)\subseteq\Ad(k)$. Therefore $z\phi_k\psi_k=(z\phi)\psi$. In particular, $z\phi_k(\phi^{-1})_k=z\phi\phi^{-1}=z=z(\phi^{-1})_k\phi_k$, so 
 $\phi_k$ has
inverse $(\phi^{-1})_k$ and 
 is therefore an automorphism. 
 Moreover, for $z\notin\Ad(k)$ we have $z\phi_k\psi_k=z=z(\phi\psi)_k$ and for $z\in \Ad(k)$ we have shown that 
$z\phi_k\psi_k=z(\phi\psi)=z(\phi\psi)_k$; so  the map $s_k:\phi\mapsto \phi_k$ is an  
endomorphism of $\St(\cK)$. By definition $\phi_k\in \St_k(\cK)$. 
 If $\a\in \St_{k}(\cK)$ then $\a s_{k}=\a_k=\a$, so $s_{k}$ has image 
$\St_{k}(\cK)$ which is a retract of $\St(\cK)$, as claimed. 
\end{proof}

\begin{corol}\label{cor:stdecomp}
Let $1\le k\le h_\cK$. The restriction 
 of $s_{k-1}:\St\maps \St_{k-1}$ to $\St_{k}$ is 
a retraction onto $\St_{k-1}$ with kernel $\St_{k}^\v$. Therefore $\St_{k}=  \St_{k}^\v\rtimes \St_{k-1}$. 
\end{corol}
\begin{proof}
As $\Ad(k-1)\subseteq \Ad(k)$, we have  $\St_{k-1}\le \St_{k}\le \St$ and as $s_{k-1}$ is a retraction of $\St$ onto $\St_{k-1}$,  the
restriction of $s_{k-1}$ to $\St_{k}$ is surjective, and so also a retraction. Denote this
restriction by $s^k_{k-1}$. Since 
$\Ad(k)=\v(k)\sqcup \Ad(k-1)$, from the definition, $\ker(s^k_{k-1})=
\St^\v_{k}$. 
\end{proof}
\begin{proof}[Proof of Theorem \ref{thm:st2decomp}]
\ref{it:st2decomp2} follows from Corollary \ref{cor:stdecomp}. 
To prove 
\ref{it:st2decomp4} first note that if $y,z\in \cmp(k)$ with $y\neq z$ then 
$[y]\cap [z]=\nul$, so  $\St_{y}^\v\cap \St_{z}^\v=1$. 
We claim next that, in this case, 
$[\St_{y}^\v,\St_{z}^\v]=1$. Suppose that $\phi_y\in\St_{y}^\v$
and $\phi_z\in  \St_{z}^\v$. Then for $u\in[y]$ there is $w\in \GG(\ad(y))$
such that $u\phi_y=w$. We have, from Lemma \ref{lem:ad0} \ref{it:ad8} and \ref{it:adcl1},  $[z]\cap\ad(y)=\nul$, so $v\phi_z=v$, for
all $v\in \supp(w)$, and thus $w\phi_z=w$ and $u\phi_y\phi_z=w$. On the other hand $u\phi_z\phi_y=u\phi_y
=w$. Similarly, for $u\in [z]$ we have $u\phi_z\phi_y=u\phi_y\phi_z$. 
For all other $u\in X$ both $u\phi_y=u$ and $u\phi_z=u$, so 
$[\phi_y,\phi_z]=1$, as claimed. Therefore  $\prod_{y\in \cmp(k)}\St_{y}^\v\le \St_{k}^\v$. 

Now let $\phi\in \St_{k}^\v$ and let $y\in \cmp(k)$. 
Define a map $\phi_y$ from $X$ to $\GG$ by
\begin{equation*}
z\phi_y=
\left\{
\begin{array}{ll}
z\phi, & \textrm{ if } z\in[y] 
\\
z, &\textrm{ otherwise}. 
\end{array}
\right.
\end{equation*} 
As in the case of $\phi_k$ in the proof of Lemma \ref{lem:st2}, to see that $\phi_y$ is a homomorphism, 
we need only check that if $a,b\in X$ such that 
$[a,b]=1$, $a\in [y]$ and $b\notin [y]$ then $[a\phi_y,b]=1$. In this situation,  since $a\in [y]$ implies $\ad(a)=\ad(y)$, we
have $[\GG(\ad(y)),b]=1$, and since $\phi\in \St_k^\v$ we have $a\phi_y=a\phi\in \GG(\ad(y))$, so  $[a\phi_y,b]=1$, as required.
 Moreover, if $\phi,\psi\in \St_k^\v$ and $a\in [y]$, then $a\phi_y\psi_y=(a\phi)\psi_y$. For $u\in [y]$ we have $u\psi_y=u\psi$, while for 
$u\in \ad(y)\bs [y]$, since $h(u)<h(y)$, we have $u\psi=u=u\psi_y$. Hence for all $u\in \ad(y)$ we have $u\psi_y=u\psi$ and therefore $(a\phi)\psi_y=a\phi\psi$. 
It follows, as in the proof of  Lemma \ref{lem:st2}, that $\phi_y$ is an automorphism of $\GG$, 
 and by definition $\phi_y\in \St_y^\V$. 

Define $\phi'=\prod_{y\in \cmp(k)} \phi_y$ and for $z\in \v(k)$ let 
$\bar z$ denote the unique element of $\cmp(k)$ such that $z\sim \bar z$. From the remark above Example \ref{ex:admiss}, 
$\v(k)=\bigcup_{y\in \cmp(k)}[y]$   so, 
for all $z\in X$,  
\begin{equation*}
z\phi'=
\left\{
\begin{array}{ll}
z\phi_{\bar z}=z\phi, & \textrm{ if } z\in \v(k) 
\\
z, &\textrm{ otherwise}. 
\end{array}
\right.
\end{equation*}
Since $\phi\in \St_k^\v$ it follows that $\phi=\phi'$. 
Therefore $\St_{k}^\v\le \prod_{y\in \cmp(k)}\St_{y}^\v$. 
\end{proof}
\begin{example}\label{ex:admiss0}
Continuing Example \ref{ex:admiss}, we have 
\[\St(\cK)=\St_3^\v\rtimes(\St_2^\v\rtimes (\St_{1}^\v\rtimes \St_{0}^\v)),\]
where 
\begin{align*}
\St^\v_3 & =\St^\v_i,\\
\St^\v_2 & =\St^\v_c\times \St^\v_f, \\
\St^\v_1 & =\St^\v_a\times \St^\v_e\textrm{ and }\\
\St^\v_0 & =\St^\v_d\times \St^\v_h.
\end{align*}
\end{example}

\section{The structure of $\St_{x}^\v$}\label{sec:stxv}
As pointed out in the introduction, the structure of $\St_{x}^\v$ depends on the
difference between $\cl(x)$ and $\ad(x)$.
\subsection{$\cl(x)=\ad(x)$}\label{sec:stxvc}
In case  $\cl(x)=\ad(x)$ a description of the structure of $\St^\v_x$ may be obtained by applying 
 the results of  \cite{DKR3}, 
where the analogue of $\St(\cK)$, for sets $\cl(x)$ instead of $\ad(x)$ is investigated.  In more detail, 
in \cite{DKR3}, the subgroups \[\St(\cL)=\{\phi\in\Aut(\GG)\,|\, \GG(\cl(x))\phi=\GG(\cl(x)), \textrm{ for all }x\in X\}\] and 
\[\St^{\conj}(\cL)=\{\phi\in\Aut(\GG)\,|\,\forall x\in X, \exists  
g_x\in \GG, \textrm{ such that } \GG(\cl(x))\phi=\GG(\cl(x))^{g_x}\}
\]
of $\Aut(\GG)$ are defined and  
\begin{itemize}
\item  it is shown  that $\St^{\conj}(\cL)=\St(\cL)\ltimes \Conj(\GG)$ \cite[Theorem 2.20]{DKR3}; and 
\item in Section 2.6,  that $\St(\cL)$ is isomorphic to a subgroup of 
$\GL(|X|,\ZZ)$ generated by upper block-triangular matrices, with diagonal blocks corresponding to the equivalence classes 
$[x]$, of elements of $x\in X$,  
together with  
a subgroup of the unipotent upper triangular matrices $U(|X|,\ZZ)$,  of nilpotency class equal to the centraliser dimension of $\GG$.
\item 
For $x\in X$ and $\phi\in \St(\cL)$, the restriction of $\phi$ to $\GG(\cl(x))$ is an automorphism of $\GG(\cl(x))$ denoted $\phi_x$. 
Define the subgroup \[\St_x(\cL)=\{\phi_x|\phi \in \St(\cL)\}\] of $\Aut(\GG(\cl(x)))$. Then the map $\rho_x:\St(\cL)\maps\St_x(\cL)$ sending $\phi$ to $\phi_x$, is 
a surjective homomorphism \cite[Lemma 2.15]{DKR3}.  
\end{itemize}
As $\cl(x)$ is a simplex, $\GG(\cl(x))$ is finitely generated free Abelian of rank $|\cl(x)|$;
 and for all $y\in \cl(x)$, we have $[y]=[y]^{\st}$.  Let $\cl(x)=\{x_1,\cdots,x_r\}$ and 
$[x]=\{x_i\,|1\le i\le s\}$, where $s\le r$. If $\phi\in \St_x(\cL)$ then, for $1\le i\le r$, we have
 $x_i=x_1^{a_{i,1}}\cdots x_r^{a_{i,r}}$, for some integers $a_{i,j}$. Therefore $\phi$ corresponds to the $r\times r$ integer 
matrix $[\phi]=(a_{i,j})$, when matrices act on the right on row vectors. Moreover, as shown in \cite{DKR3}, $[\phi]$ is an upper  block-triangular matrix, with diagonal blocks corresponding
to the equivalence classes of elements of $\cl(x)$. More precisely, let $\{y_i\,|\,1\le i\le m\}$ be a  transversal for the equivalence relation
$\sim$ restricted to $\cl(x)$ (and assume $y_1=x$). 
Then $\cl(x)=\cup_{i=1}^m [y_i]$ and $[\phi]$ has 
\be[label=(\roman*),ref=(\roman*)]
\item\label{it:astx1} $m$ diagonal blocks $A_1,\ldots, A_m$, 
where $A_i\in \GL(|[y_i]|,\ZZ)$ and 
\item\label{it:astx2} $a_{i,j}=0$ if $i>j$ and $a_{i,j}$ is not in the $i$th block $A_i$ of $[\phi]$. 
\ee
Let $\cS_x$ denote the set of matrices satisfying the
two conditions above.  
From \cite[Lemma 2.15]{DKR3},  the map $\pi_x$  
such that $\phi\pi_x = [\phi]$ is
 a isomorphism from $\St_x(\cL)$ to the subgroup $\cS_x$ of $\GL(r,\ZZ)$. 

  Also, writing $A=[\phi]$ and $A_D$ for the block-diagonal matrix which has diagonal blocks $A_1,\ldots ,A_m$ and zeros elsewhere, 
we have $A_D\in \prod_{i=1}^m\GL(|[y_i]|,\ZZ)$ and $A_D^{-1}A$ is a unipotent upper triangular matrix $A_U$: that is an element of $U(r,\ZZ)$,
satisfying \ref{it:astx1} and \ref{it:astx2} above, but with $A_i$ equal to the identity matrix, for all $i$. It follows that 
 $(A_U-I)^m=0$, so the subgroup  $\cS_U=\{A_U\,|\, A=[\phi], \phi\in \St_x\}$ is a nilpotent subgroup of $\GL(r,\ZZ)$ of class $m-1$. 
Furthermore (\cite[Lemma 2.18]{DKR3}) setting $\cS_D=\{A_D\,|\, A=[\phi], \phi\in \St_x\}$, we have $\cS_x=\cS_D\ltimes \cS_U$ with 
$\cS_D=\prod_{i=1}^m\GL(|[y_i]|,\ZZ)$. (\emph{Errata:}  
In  
 \cite[Lemma 2.18]{DKR3}, the equality for $D_Y$ should be $D_Y=\prod_{i=1}^m\GL(|[v_i]_{\perp}|,\ZZ)$ in both statement and proof.)


Note that, for all $z\in X$, 
 $y\in \ad(z)$ implies $y\in \cl(y)\subseteq \ad(y)\subseteq \ad(z)$, so  $\ad(z)=\cup_{y\in \ad(z)}\cl(y)$, and it follows that 
$\St(\cL)\subseteq \St(\cK)$. 
 Returning to $x$ such that $\cl(x)=\ad(x)$, we claim that in this case $\St_x^\v(\cK)$ is a subgroup of $\St(\cL)$.  
To see this, suppose that $\phi\in \St_x^\v(\cK)$ and $y\in X$. If $y\in \ad(x)$ then $\ad(y)\subseteq \ad(x)$, which is 
a simplex, from which it follows that $\ad(y)=\cl(y)$. Hence, as $\phi\in \St(\cK)$, we have $\GG(\cl(y))\phi=\GG(\ad(y))\phi=\GG(\ad(y))
=\GG(\cl(y))$. On the other hand if $y\notin \ad(x)$ let $u\in \cl(y)$. If $u\notin [x]$ then $u\phi=u\in \cl(y)$. If $u\in [x]$ then
$[x]\subseteq \cl(y)$ so $u\in \ad(x)=\cl(x)\subseteq \cl(y)$. Hence $u\phi\in \GG(\ad(x))\subseteq \GG(\cl(y))$. In both cases
 $u\phi\in \GG(\cl(y))$. The same arguments apply to $\phi^{-1}$, and it follows that $\GG(\cl(y))\phi=\GG(\cl(y))$, completing the proof 
of the claim. 

As $\St_x^\v(\cK)\le \St(\cL)$ we may consider the restriction of the homomorphism $\rho_x$ above to $\St_x^\v$. This restriction maps
$\St_x^\v$ isomorphically to its image in $\St_x$. Indeed, if $\phi,\phi'\in  \St_x^\v$ are such that $\phi\rho_x=\phi'\rho_x$ then
$y\phi\rho_x=y\phi'\rho_x$, for all $y\in \cl(x)$, so $y\phi\rho_x=y\phi'\rho_x$ for all $y\in [x]$, and it follows that $\phi=\phi'$. 
Therefore, the composition $\rho_x\pi_x$ maps $\St_x^\v$ isomorphically to its image in $\cS_x$, which we call $\cS_x^\v$. Now 
let $\phi\in \St_x^\v$, let $[\phi]=\phi\rho_x\pi_x$ and write $[\phi]=(a_{i,j})_{i,j=1}^r$. Then $(a_{i,j})_{i,j=1}^r$ satisfies satisfies \ref{it:astx1} and
\ref{it:astx2} above. Moreover as 
%
$\phi\in \St_x^\v$, for $i>s$, we have $x_i\phi=x_i$, 
 so $a_{i,i}=1$ and $a_{i,j}=0$, for $i\neq j$. Hence $A_i$ is the identity matrix for $1<i\le m$. Thus $(a_{i,j})_{i,j=1}^r$ satisfies 
\be[resume,label=(\roman*),ref=(\roman*)]
\item\label{it:vastx1} $A_1=(a_{i,j})_{i,j=1}^s$ is in $\GL(|[x]|,\ZZ)$;
\item\label{it:vastx2} $a_{i,i}=1$, for $i>s$, and 
\item\label{it:vastx3} 
$a_{i,j}=0$ if $i\neq j$  and $i>s$. 
\ee
Conversely, any matrix $A$ satisfying \ref{it:vastx1}, \ref{it:vastx2} and \ref{it:vastx3} determines a unique element
$A\pi_x^{-1}\rho_x^{-1}$ of $\St_x^\v$. 
 If $A\in \cS_x^\v$ then the matrix $A_D$ obtained from $A$ by setting $a_{i,j}=0$, for $(i,j)$ such that $1\le i\le s$ and $j>s$, 
is uniquely determined by the block $A_1$, while $A_U=A_{D}^{-1}A$ satisfies $(A_U-I)^2=0$. This gives the following theorem, in which
for $1\le i, j\le n$ and $i\neq j$, 
\begin{itemize}
\item
$E^n_{i,j}$ is an $n\times n$ square matrix, with $1$'s on the leading diagonal, 
a $1$ in position $i,j$, and zeros elsewhere; and 
\item $O^n_i$ is 
an $n\times n$ square diagonal matrix with $1$'s on the leading diagonal except for row $i$ which has diagonal entry $-1$ (and zeros off the leading diagonal). 
\item For $m\le n$, $\cM(m,n-m)$ is the group of $m\times (n-m)$ integer matrices under addition, and $Z_{i,j}\in \cM(m,n-m)$ is the matrix with every coefficient
equal to $0$, except the $(i,j)$ coefficient which is equal to $1$. 
\end{itemize}
Since the Whitehead automorphisms of Type 2 involved here are all transvections we use the Laurence-Servatius notation for generators in this case: 
that is, in 
the terminology of Section \ref{sec:genaut} we use $\tr_{x,y}$ rather than $(\{x,y\},y)$, to denote the transvection mapping $x$ to $xy$. 

\begin{theorem}\label{thm:stxva}
  Let $x\in X$ such that $\ad(x)=\cl(x)$. Assume  and  that
  $\ad(x)=\{x_i\,|\,1\le i\le r\}$ and $[x]=\{x_i\,|\,1\le i\le s\}$, where $s\le r$. Then 
\[\St_x^{\v}(\cK)= \St_{x,D}^\v\ltimes \St_{x,U}^\v\cong \GL(s,\ZZ)\ltimes_\theta \cM(s,r-s),\]
where, 
\be[label=(\roman*),ref=(\roman*)]
\item\label{it:stxva1} $\St_{x,U}^\v$ is free Abelian of rank $s(r-s)$, freely generated by the set of automorphisms 
$\{\tr_{x_i,x_j}\,|\,1\le i\le s, s+1\le j\le r\}$, and is isomorphic to $\cM(s,r-s)$ by an isomorphism taking $\tr_{x_i,x_j}$ to $Z_{i,j-s}$; 
\item\label{it:stxva2}  $\St_{x,D}^\v$ 
is generated by $\{\tr_{x_i,x_j}, \i_{x_i}\,|\, 1\le i,j\le s, i\neq j\}$ and the map $\tr_{x_i,x_j}\mapsto E^s_{i,j}$, $\i_{x_i}\mapsto O^s_i$,  
where $E^s_{i,j}$ and $O^s_i$ are the $s\times s$ matrices above,
extends to an isomorphism  $\St_{x,D}^\v$ to $\GL(s,\ZZ)$; and
\item\label{it:stxva3}  for $A\in \GL(s,\ZZ)$ and $B\in \cM(s,r-s)$, the automorphism $A\theta$ maps $B$ to $A^{-1}B\in  \cM(s,r-s)$.
\ee 
\end{theorem}
\begin{proof}
We have established, using results of \cite{DKR3}, that $\St_x^\v\cong \cS_x^\v\le \GL(r,\ZZ)$ via the isomorphism $\rho_x\pi_x$, such that 
 $\phi\mapsto [\phi]$, for $\phi\in \St_x^\v$.
The subgroup $\cS_x^\v$ consists of matrices satisfying   \ref{it:vastx1}, \ref{it:vastx2} and \ref{it:vastx3} above, so if $A\in \cS_x^\v$ 
then, as above,  we may write
$A=A_DA_U$,  where
\[A_U= \begin{bmatrix}
    I_s     & A_U' \\
    0  & I_{r-s} 
\end{bmatrix},\, 
A_D=\begin{bmatrix}
    A_D'    & 0 \\
    0  & I_{r-s} 
\end{bmatrix},\]
with $A_U'\in \cM(s,r-s)$, $A_D'\in \GL(s,\ZZ)$, and $I_s$ and $I_{r-s}$ the identity matrices of the appropriate dimensions.

Define  subgroups $\cS_{x,D}^\v=\{A_D\,|\,A\in \cS_x^\v\}$ and 
$\cS_{x,U}^\v=\{A_U\,|\,A\in \cS_x^\v\}$. If 
$U,V\in \cS_{x,U}^\v$ and  $W\in \cS_{x,D}^\v$ with  
\[U= \begin{bmatrix}
    I_s     & U' \\
    0  & I_{r-s} 
\end{bmatrix},\, 
V=\begin{bmatrix}
    I_s     & V' \\
    0  & I_{r-s} 
\end{bmatrix},\, 
W=\begin{bmatrix}
    W'    & 0 \\
    0  & I_{r-s} 
\end{bmatrix},
\]
then 
\[UV= 
\begin{bmatrix}
    I_s     & U'+V' \\
    0  & I_{r-s} 
\end{bmatrix}
=VU\textrm{ and }W^{-1}UW= 
\begin{bmatrix}
    I_s     & W'^{-1}U' \\
    0  & I_{r-s} 
\end{bmatrix}
\in \cS_{x,U}^\v
\]
and we deduce that $\cS_{x,U}^\v$ is free Abelian and 
$\cS_{x}^\V= \cS_{x,D}^\v\ltimes \cS_{x,U}^\v$.
Also the map $\pi_U$, sending $U$ above to $U' \in \cM(s,r-s)$, and the map $\pi_D$, sending $W$ above to $W'$ in $\GL(s,\ZZ)$ are isomorphisms
from $\cS_{x,U}^\v$ to $\cM(s,r-s)$ and from  $\cS_{x,D}^\v$ to $\GL(s,\ZZ)$, respectively. 
 
Let $\Tr_{x,U}$ denote the set of transvections  $\{\t_{x_i,x_j}\,|\,1\le i\le s, s+1\le j\le r\}$ and define
 $\St^\v_{x,U}$ to be the subgroup of $\St^\v_x$ generated by $T_{x,U}$. Elements of $\St^\v_{x,U}$ fix the 
set $\{x_i\,|\, 1\le i\le s\}$ point-wise, and the map $\rho_x\pi_x\pi_U$ maps  $\t_{x_i,x_j}$ to $Z_{i,j-s}\in \cM(s,r-s)$, for
all $\tr_{x_i,x_j}\in \Tr_{x,U}$. It follows that $\rho_x\pi_x\pi_U$ maps $\St^\v_{x,U}$   isomorphically to   
 $\cM(s,r-s)$, and as the latter is freely generated by the $Z_{i,j}$, $1\le i\le s$, $1\le j\le r-s$, this proves \ref{it:stxva1}.
Similarly,  $\rho_x\pi_x\pi_D$
 sends $\t_{x_i,x_j}$ to $E^s_{i,j}$ and $\i_{x_i}$ to $O^s_{i}$, for $1\le i,j\le s$, and $i\neq j$; so determines an isomorphism from $\St_{x,D}^\v$ to 
$\GL(s,\ZZ)$. As the latter is generated by the matrices $E^s_{i,j}$ and $O^s_i$, \ref{it:stxva2} follows.
Finally, \ref{it:stxva3} follows from the identity for $W^{-1}UW$ above. 
\end{proof}

\begin{example}\label{ex:admiss1}
Continuing Example \ref {ex:admiss};  we have $\ad(x)=\cl(x)$ for $x=d,e,f$ and $h$. 
\be
\item $\ad(f)=\cl(f)=\{d,e,f,g\}$ and $[f]=\{f,g\}$, so $r=4$ and $s=2$.

$\St_{f,U}^\v=\la \tr_{f,d}, \tr_{f,e},\tr_{g,d}, \tr_{g,e}\ra$ and is isomorphic to $\cM(2,2)$ via the map sending 
$\tr_{f,d}, \tr_{f,e},\tr_{g,d}$ and $\tr_{g,e}$ to $Z_{1,1}$,  $Z_{1,2}$, $Z_{2,1}$ and $Z_{2,2}$, respectively. 

$\St_{f,D}^\v=\la  \tr_{f,g}, \tr_{g,f}, \i_f, \i_g\ra$ and is isomorphic to $\GL(2,\ZZ)$ via the map sending 
$\tr_{f,g}, \tr_{g,f},\i_f$ and $\i_g$ to $E_{1,2}^2$, $E_{2,1}^2$, $O^2_1$ and $O^2_2$. 

Combining these two isomorphisms, $\St_f^\v$ is isomorphic to $\GL(2,\ZZ)\ltimes_\theta \cM(2,2)$, where, for $A\in \GL(2,\ZZ)$, the automorphism $A\theta$ of $\cM(2,2)$ is the 
map sending $B$ to $A^{-1}B$, for $B\in \cM(2,2)$. 
\item In the same way we see that $\St^\v_e= \la \i_e\,|\, \i_e^2\ra\ltimes\la \t_{e,d}\,|\,\ra\cong D_\infty$, the infinite dihedral group. 
\item Similar considerations show that $\St^\v_h=\la \i_h\,|\,\i_h^2\ra\cong \ZZ_2$ and 
 $\St^\v_d=\la \i_d\,|\,\i_d^2\ra\cong \ZZ_2$. 
\ee
\end{example}
\subsection{$\cl(x)\neq \ad(x)$}\label{sec:stxva}
In this case $\cl(x)$ is a proper subset of $\ad(x)$ and we define  $\ad_\out(x)=\ad(x)\bs \cl(x)$ and $\ad_\innt(x)=\cl(x)\bs [x]$. 
Then we have a disjoint union 
$\ad(x)=[x]\sqcup \ad_\innt(x)\sqcup \ad_\out(x)$, with $\ad_\innt(x)\subseteq \lk(x)$ and $\ad_\out(x)\cap \st(x)=\nul$. 
Let $\ad(x)=\{x_i\,|\, 1\le i\le r\}$, where for some $p\le q<r$ we have  $[x]=\{x_i\,|\, 1\le i\le p\}$ and $\ad_\innt(x)=\{x_i\,|\,p+1\le i\le q\}$.
Then 
$\GG(\ad(x))=\GG(\ad_\innt(x))\times [\GG([x])\ast \GG(\ad_\out(x))]$, where  $\GG([x])$ is free of rank $p$,  
$\GG(\ad_\innt(x))$ is isomorphic to  $\ZZ^{q-p}$ and $\GG(\ad_\out(x))$ is a partially commutative group on a graph of $r-q$ vertices. 
\begin{theorem}\label{thm:stxdecomp}
$\St_x^\v= \St_{x,l}^\v \ltimes \St_{x,s}^\v$ where 
\[\St_{x,l}^\v=\{\phi\in \St_x^\v\,|\, y\phi \in \GG([x]\cup \ad_\out(x)), \forall y\in [x]\}\] and
\[\St_{x,s}^\v=\{\phi\in \St_x^\v\,|\, \forall y\in [x], \exists w_y\in \GG(\ad_\innt(x))\textrm{ such that } y\phi=yw_y\}.\]
Moreover,  with the above notation, 
$\St_{x,s}^\v$ is a free Abelian group of rank $p(q-p)$, freely generated by $\{\t_{x_i,x_j}\,|\, 1\le i\le p, p+1\le j\le q\}$,
 and is isomorphic to $\cM(p,q-p)$ by an isomorphism taking $\tr_{x_i,x_j}$ to $Z_{i,j-p}$ (\textit{cf.} Theorem \ref{thm:stxva}). 
\end{theorem}
\begin{proof}
  Let $\phi\in \St_x^\v$, so $y\phi=y$ unless $y\in \{x_1,\ldots ,x_p\}$.
  For $1\le i\le p$, there exists $w_i\in \GG([x]\cup \ad_\out(x))$ such that  
\[x_i\phi=w_ix_{p+1}^{a_{i,p+1}}\cdots x_{q}^{a_{i,q}},\] as $\GG(\ad_\innt(x))$ is the centre of $\GG(\ad(x))$. 
Let \[\phi_1=\prod_{i=1}^p\tr_{x_i,x_{p+1}}^{a_{i,p+1}}\cdots\tr_{x_i,x_{q}}^{a_{i,q}},\] 
so $x_i\phi_1=x_ix_{p+1}^{a_{i,p+1}}\cdots x_{q}^{a_{i,q}}$, for $i=1,\ldots ,p$. 
Then, as $\phi_1\in \St_x^\v$, so is $\phi_0=\phi_1^{-1}\phi$ and $x_i\phi_0=w_i\in \GG([x]\cup \ad_\out(x))$, for $x_i\in [x]$. Therefore
$\phi=\phi_1\phi_0$ with $\phi_0\in \St_{x,l}^\v$ and  $\phi_1\in \St_{x,s}^\v$. Moreover, as in the previous subsection, $\St_{x,s}^\v$ is 
a free Abelian group generated by $\{\tr_{x_i,x_j}\,|\, 1\le i\le p, \, p+1\le j\le q\}$, isomorphic to $\cM(p,q-p)$ via the map sending
$\t_{x_i,x_j}$ to $Z_{i,j-p}$. 
 From the definitions $\St_{x,s}^\v\cap \St_{x,l}^\v=\{1\}$. 
To see that  $\St_{x,s}^\v$ is normal in $\St_x^\v$, let  $\tr_{x_k,x_j}$ be a generator of $\St_{x,s}^\v$ and $\phi\in \St_{x,l}^\v$. Then, for $x_i\in [x]$ 
there exists $w_i\in \GG([x]\cup \ad_{\out}(x))$ such that $x_i\phi^{-1}=w_i$. For each $i$, let $s(i,k)$ be the 
exponent sum of $x_k$ in $w_i$. As $x_k\in [x]$ and $x_j\in \ad_s(x)$, we have  $x_i\phi^{-1}\tr_{x_k,x_j}\phi=w_i\tr_{x_k,x_j}\phi=w_ix_j^{s(i,k)}\phi=
w_i\phi x_j^{s(i,k)}=x_ix_j^{s(i,k)}$. Therefore $\phi^{-1}\tr_{x_k,x_j}\phi\in \St_{x,s}^\v$; from which it follows  that $\St_{x,s}^\v$ is normal in $\St_x^\v$.  
\end{proof}
\begin{example}\label{ex:admiss2}
Continuing Example \ref {ex:admiss}; $\ad(x)\neq \cl(x)$ when $x=a, b, c$ or $i$. 
We have $[a]=\{a,b\}$ and  $\ad_\innt(a)=\{d\}$ so $\St^\v_{a,s}=\la \tr_{a,d},\tr_{b,d}\ra$ is free Abelian of rank $2$. 
As $[c]=\{c\}$ and $\ad_\innt(c)=\{d\}$, we have $\St^\v_{c,s}=\la \tr_{c,d}\ra$ and similarly $\St^v_{i,s}=\la \tr_{i,h}\ra$, both infinite cyclic.  
\end{example}
This lemma allows us to reduce determination of the structure of $\St_{x}^\v$ to that of $\St_{x,l}^\v$. For this purpose it is  convenient
to consider the set of Whitehead automorphisms $\W$ as in Definition \ref{def:wh}, as the generating set for $\Aut(\GG)$.  
 Our candidate generating set for $\St_{x,l}^\v$ is given in the next definition. 
\begin{definition}\label{defn:stx1gen}
Given $x\in X$, let $\W_x=\W\cap \St_{x,l}^\v$. 
\end{definition}
From the definitions we have
\begin{align*}
\W_x&=\{\s\in \Aut(\G^{\pm})\,|\, y\s=y,\forall y \in X\bs [x]\}\\
&\cup \{(A,a)\in \W\,|\,A\bs\{a\}\subseteq [x]_L, a\in[x]_L\cup \ad_{\out,L}(x)\}.
\end{align*}
As the full graph on $[x]$ is a null graph in the current case,  the set of Type 1 automorphisms in $\W_x$ is  
$\{\s\in \Aut(\G^{\pm})\,|\, y\s=y,\forall y \in X\bs [x]\}$, which is the set of permutations $\s$ of $[x]_L$ such that $x^{-1}\s=(x\s)^{-1}$.  

As a candidate set of relations $R_x$  for $\St_{x,l}^\v$ we take those relations of the presentation
for $\Aut(\GG)$ in \cite{Day09} which apply to words in the free group on $\W_x$; augmented by relations required to make
peak reduction arguments possible within the set $\W_x$ (namely \ref{it:DR3*} and \ref{it:DR4*}). More precisely we define the set $R_x$ to consist 
of all relations defined by \ref{it:DR1}--\ref{it:DR7}, \ref{it:DR3*} and \ref{it:DR4*} below. In these relations, $A+B$ denotes $A\cup B$, when $A\cap B=\nul$, 
and $B-A$ denotes $A\backslash B$, when $A\subseteq B$. By $A-a$ and $A+a$ we mean $A- \{a\}$ and $A+ \{a\}$ respectively. 
\be[label=R\arabic*$_x$,ref=R\arabic*$_x$]
\item \label{it:DR1}
$(A,a)^{-1}=(A-a+a^{-1},a^{-1})$, for $(A,a)\in \W_x$. 
\item \label{it:DR2}
$(A,a)(B,a)=(A\cup B, a)$, for $(A,a), (B,a)\in \W_x$, such that $A\cap B=\{a\}$.
\item \label{it:DR3}
$(B,b)^{-1}(A,a)(B,b)=(A, a)$, for $(A,a), (B,b)\in \W_x$, such that $a^{-1}\notin B$, $b^{-1} \notin A$ 
and either 
\be[label=(\alph*),ref=(\alph*)]
\item \label{it:DR3a}
$A\cap B=\nul$ or 
\item \label{it:DR3b} $a\in \lk_L(b)$.
\ee
\ee
\be[label=R\arabic*$_x^*$,ref=R\arabic*$_x^*$,start=3]
\item \label{it:DR3*}
$(B,b)^{-1}(A,a)(B,b)=(A, a)$, for $(A,a), (B,b)\in \W_x$, such that $a^{-1}\in B$, $b \notin A$ 
and 
$A\subseteq B$. 
\ee
\be[label=R\arabic*$_x$,ref=R\arabic*$_x$,start=4]
\item \label{it:DR4}
$(B,b)^{-1}(A,a)(B,b)=(A+B-b,a)$, 
for $(A,a)$, $(B,b)\in \W_x$, such that $a^{-1}\notin B$, $b^{-1}\in  A$, 
and 
$A\cap B=\nul$. 
\ee
\be[label=R\arabic*$_x^*$,ref=R\arabic*$_x^*$,start=4]
\item \label{it:DR4*}
$(B,b)^{-1}(A,a)(B,b)=(B-A+b^{-1},a^{-1})$, for $(A,a)$, $(B,b)\in \W_x$, such that $a^{-1}\in B$, $b\in A$, 
and 
$A\subseteq B$.
\ee
\be[label=R\arabic*$_x$,ref=R\arabic*$_x$,start=5]
\item \label{it:DR5}
$(A,a)(A-a+a^{-1},b)=\s_{a,b}(A-b+b^{-1},a)$, for $(A,a)\in \W_x$, $a,b\in [x]_L$, $a\neq b$, $b\in A$, $b^{-1}\notin A$  and $\s_{a,b}$ the
Type 1 Whitehead automorphism permuting $[x]_L$ by the cycle $(a,b^{-1},a^{-1},b)$.  
\item \label{it:DR6}
$\s^{-1}(A,a)\s=(A\s,a\s)$, where $(A,a)$ is in $\W_x$, of Type 2, and $\s \in \W_x$ of Type 1. 
\item \label{it:DR7}
The multiplication table of the subgroup  of Type 1 automorphisms in $\W_x$. 
\ee
We denote by R$_x$ the set of  of relations given by \ref{it:DR1}--\ref{it:DR7}, \ref{it:DR3*} and \ref{it:DR4*}.
In the remainder of this section we shall prove the following theorem.
\begin{theorem}\label{thm:stxvf}
  $\St_{x,l}^\v$ has a presentation $\la \W_x|R_x\ra$.
\end{theorem}

We shall use the peak reduction theorem of \cite{Day09}, and its analogue for $\St_{x,l}^\v$, to prove this theorem, 
and introduce the necessary 
terminology in the next sub-section. First we prove Corollary \ref{cor:stgens}
and given an example. 
\begin{proof}[Proof of Corollary \ref{cor:stgens}.]
  It must be shown that $\St(\cK)=\la \Inv,\Tr\ra$; and that there is a
  finite presentation with 
  these generators there is a 
Every inversion and elementary transvection belongs to $\St(\cK)$, by the fundamental results of Laurence and Servatius. 
 On the other hand, it follows from Theorems \ref{thm:st2decomp},  \ref{thm:stxva},  \ref{thm:stxdecomp} and   \ref{thm:stxvf}
 that $\St(\cK)$ is generated by inversions and transvections. Moreover,
 from the constructions appearing in these theorems a finite presentation
 may be built. 
\end{proof}
\begin{example}\label{ex:admiss3}
  Continuing from Example \ref{ex:admiss2} we find presentations for $\St^\v_{x,l}$, for $x=a,c,i$. 
\be
\item We have $[i]=\{i\}$, $\ad_s(i)=\{h\}$ and $\ad_l(i)=\{c,d,e\}$. 
Then 
\[\W_i=\{\i_i, (\{i^\e,s\},s), (\{i,i^{-1},s\},s)\,:\, \e\in \{\pm 1\}, s\in \{c,d,e\}^{\pm 1}\}.\]
The set $R_i$ consists of relations of types $R1_i$, $R2_i$, $R3_i$, $R6_i$ and $R7_i$. There are no
relations of types $R3^*_i$, $R4_i$, $R4^*_i$ or $R5_i$. Relations $R1_i$ and $R2_i$ allow Tietze transformations
to be applied to remove generators $(\{i^\e,s^{-1}\},s^{-1})$, where $s\in \{c,d,e\}$, and generators $(\{i,i^{-1},s\},s)$, where 
$s\in \{c,d,e\}^{\pm 1}$. This leaves a presentation with generating set  
\[\W'_i=\{\i_i, (\{i^\e,s\},s)\,:\, \e\in \{\pm 1\}, s\in \{c,d,e\}\},\]
and relations
\be[label=R\arabic*$_i$.,ref=R\arabic*$_i$]
\item[R$2_i+$R$3_i$(a).] $(\{i,s\},s)(\{i^{-1},s'\},s')=(\{i^{-1},s'\},s')(\{i,s\},s)$,  $s,s'\in \{c,d,e\}$.
\item [R$3_i$(b).] $(\{i^\e,d\},d)(\{i^{\e},s\},s)=(\{i^{\e},s\},s)(\{i^{\e},d\},d)$,  $s\in \{c,e\}$, $\e\in \{\pm 1\}$.
\addtocounter{enumii}{5}
\item $\i_i(\{i,s\},s)\i_i=(\{i^{-1},s\},s)$, $s\in \{c,d,e\}$.
\item $\i_i^2=1$. 
\ee

For each $\e=1$ and $-1$ we have a subgroup of $\St^\v_{i,l}$ generated by $\{(\{i^\e,s\},s):s\in  \{c,d,e\}\}$ which is 
isomorphic to $\FF_2\times \ZZ$ (the central $\ZZ$ generated by $(\{i^\e,d),d)$). The conjugation action of $\i_i$ on this subgroup maps
$(\{i^\e,s\},s)$ to $(\{i^{-\e},s\},s)$, for all $s$, so \[\St^\v_{i,l}\cong (\FF_2\times \ZZ)^2\rtimes_{\phi_i} \ZZ_2,\] where $\phi_i$ maps $1\in \ZZ_2$ to the 
automorphism of $(\FF_2\times \ZZ)^2$ taking $(a,b)$ to $(b,a)$, for $a,b\in \FF_2\times \ZZ$. 
\item  We have $[c]=\{c\}$, $\ad_s(c)=\{d\}$ and $\ad_l(c)=\{e\}$. As before, applying Tietze transformations to the presentation obtained 
from Theorem \ref{thm:stxvf} gives a presentation with generating set 
\[\W'_c=\{\i_c, (\{c^\e,e\},e)\,:\,\e=\pm 1\},\]
and relations
\be
\item[R$2_c$.] $(\{c^{-1},e\},e)(\{c,e\},e)=(\{c,e\},e)(\{c^{-1},e\},e)$.
\item[R$6_c$.] $\i_c(\{c,e\},e)\i_c=(\{c^{-1},e\},e)$.
\item[R$7_c$.] $\i_c^2=1$. 
\ee
Hence $\St^\v_{c,l}\cong \ZZ^2\rtimes_{\phi_c} \ZZ_2$, where $\phi_c$ maps $1\in \ZZ_2$ to the 
automorphism of $(\ZZ)^2$ taking $(a,b)$ to $(b,a)$, for $a,b\in \ZZ$. 
\item
We have $[a]=\{a,b\}$, $\ad_s(a)=\{d\}$ and $\ad_l(a)=\{h\}$. Let $\Pi_a$ denote the set of permutations of $\{a^{\pm 1}, b^{\pm 1}\}$ inducing 
automorphisms of $\GG(a,b)$.  We use the more concise notation $\t_{x,y}$ for the automorphism $(\{x,y\},y)$ here. After applying Tietze transformations
as in the previous cases we  obtain a presentation for $\St^\v_{a,l}$ with generating set 
\[
\W'_a= \Pi_a\cup \{\t_{s,t}\,:\, s\in \{a,b\}^{\pm 1}, t\in \{a,b,h\}, s\neq t^{\pm 1}\},
\]
and relations $R'_a$ as follows. 
\be
\item[R$2_a$.] \[\t_{s,t}\t_{s^{-1},t}=\t_{s^{-1},t}\t_{s,t},\] where $s\in \{a,b\}$, $t\in \{a,b,h\}$, $s\neq t^{\pm 1}$.
\item[R$3_a^*$.] \[\t_{s,t}\t_{s,h}\t_{t,h}\t_{t^{-1},h}= \t_{s,h}\t_{t,h}\t_{t^{-1},h}\t_{s,t},\] where  $s\in \{a,b\}^{\pm 1}$, $t\in \{a,b\}$, 
$s\neq t^{\pm 1}$.
\item[R$4_a^*$.] \[\t_{s,t}\t_{t,s}\t_{t^{-1},s}=\t_{t,s}\t_{t^{-1},s}\t_{s^{-1},t}^{-1},\] where  $s,t\in \{a,b\}$, $s\neq t$.
\item[R$5_a$.] 
\[
\t_{s^\e,t}\t_{t^{-1},s}^\e=\s_{s^\e,t}\t_{a^{-\e},b},
\]
where $\e\in \{\pm 1\}$, $s,t\in \{a,b\}$, $s\neq t^{\pm  1}$, and $\s_{s^\e,t}$ is the permutation with cycle $(s^\e,t,s^{-\e},t^{-1})$.
\item[R$6_a$.] For all $\s \in \Pi_a$, and all $\t_{s,t}$,  
\[\s^{-1}\t_{s,t}\s=\t_{s\s,t\s}.\]
\item[R$7_a$.] A set of defining relations for $\la \Pi_a\ra$.
\ee
From Examples \ref{ex:admiss0}, \ref{ex:admiss1} and \ref{ex:admiss2}, 
\begin{align*}
\St(\cK)&\cong\left[\ZZ\rtimes \St_{i,l}^\v \right]\\
&\rtimes\left\{ \left[ \left(\ZZ\rtimes \St^\v_{c,l}\right)\times \left(\cM(2,2)\rtimes \GL(2,\ZZ) \right)\right]\right.\\
&\rtimes\left\{
\left[\left(\ZZ^2\rtimes\St_{a,l}^\v \right) \times D_\infty  \right]\right.\\
&\left. \left.\rtimes\left[ \ZZ_2\times \ZZ_2\right]\right\}\right\}
\end{align*}
and combining with the current  example we have
\begin{align*}
\St(\cK)&\cong\left[\ZZ\rtimes \left(\left(\FF_2\times \ZZ\right)^2\rtimes \ZZ_2\right) \right]\\
&\rtimes\left\{ \left[ \left(\ZZ\rtimes \left(\ZZ^2\rtimes \ZZ_2 \right)\right)\times \left(\cM(2,2)\rtimes \GL(2,\ZZ) \right)\right]\right.\\
&\rtimes\left\{
\left[\left(\ZZ^2\rtimes \la \W'_a\,|\,R'_a\ra\right)\times D_\infty   \right]\right.\\
&\left. \left.\rtimes\left[ \ZZ_2\times \ZZ_2\right]\right\}\right\}.
\end{align*}
\ee
From this decomposition we could construct a presentation of $\St(cK)$, with generators $\Inv$ and $\Tr$, as in Corollary \ref{cor:stgens}.
\end{example}
\begin{remark}
  In the case when $[x]=\{x\}$ and $\ad(x)\neq \cl(x)$ the group
  $\Aut(\GG([x]))$
  is cyclic of order $2$, generated by the inversion $\i_x$ which permutes
  the elements of $\{x^{\pm 1}\}$. Every element $\phi$ of
  $\St^\v_{x,l}$ maps $x$ to a word $w_1x^\e w_2$,
  where $w_i\in \GG(\ad_{out}(x))$,and $\e=\pm 1$. It follows that
  $\St^\v_{x,l}$ is isomorphic to the wreath product $C_2\wr \GG(\ad_{\out}(x))$.
  However, when $|[x]|\ge 2$ although, similarly,  $\St^\v_{x,l}$ contains an
  subgroup $H$ 
  isomorphic to the wreath product $\textrm{Sym}(L)\wr \GG(\ad_{out}(x))$, it also
  contains elements outside $H$; for example $\tr_{x,y}\tr_{x,a}$, where
  $x,y\in [x]$ and $a\in \ad_{\out}(x)$. 
  \end{remark}
\subsection{Peak reduction in $\Aut(\GG)$}\label{sec:peak}
The \emph{length} of a conjugacy class $c$ of $\GG$ is the minimum of the lengths of words
representing elements of $c$, denoted $|c|_\sim$. 
The \emph{length} of a $k$-tuple $C=(c_1,\ldots,c_k)$ of conjugacy classes is $|C|_\sim=\sum_{i=1}^k|c_i|_\sim$. 
If $\a\in \Aut(\GG)$ and $c$ is a conjugacy class in $\GG$ then by $c\a$ we 
mean the conjugacy class of $w\a$, where $w$ is an element of $c$. If $C=(c_1,\ldots,c_k)$ is a $k$-tuple of conjugacy 
classes of $\GG$ and $\a$ is an automorphism then we write $C\a=(c_1\a,\ldots ,c_k\a)$. 
\begin{definition}\label{def:pr}
Let $\a,\b\in\W$ and let $C$ be a $k$-tuple of conjugacy classes. The composition $\a\b$ is a \emph{peak} with respect to $C$ if 
\[
|C\a|_\sim\ge |C|_\sim\textrm{ and } |C\a|_\sim\ge |C\a\b|_\sim\]
and at least one of  these inequalities is strict.  
Let $\W'$ be a subset of $\W$ and let $\a\b$ be a peak with respect to $C$. 
A \emph{peak lowering} of $\a\b$ for $C$, in $\W'$,  is a factorisation $\a\b=\d_1\cdots \d_s$, such that $\d_i\in \W'$ and 
\[|C\d_1\cdots \d_i|_\sim<|C\a|_\sim,\]
for $1\le i\le  s-1$.

Let $\phi\in \Aut(\GG)$ have factorisation $\phi=\a_1\cdots \a_m$, where $\a_i\in \W$. For $1\le i\le m-1$, this factorisation 
is said to have a \emph{peak} with respect to $C$, at $i$, if $\a_i\a_{i+1}$ is a peak with respect to $C\a_1\cdots\a_{i-1}$. 
If the factorisation has no peak with respect to $C$ it is said to be \emph{peak reduced} with respect to $C$.  
\end{definition}

Day proves \cite[Lemma 3.18]{Day09} that if $C$ is a $k$-tuple of conjugacy classes of $\GG$ and $\a,\b\in \W_l$ such that $\a\b$ is a 
peak with respect to $C$ then
there is a peak lowering of $\a\b$ with respect to $C$, in $\W_l$.  We shall first use 
the peak lowering theorem of \cite{Day09} to show that $\St^\v_{x,l}$ is generated by $\W_x$. Then we shall 
establish that  peak lowering can be carried out in the set $\W_x$, and use this to prove Theorem \ref{thm:stxvf}. 
\subsection{Generators for $\St_x^\v$}\label{sec:stxvgen}
Let $W=(w_1,\ldots, w_k)$ be a $k$-tuple of elements of $\GG$.  The \emph{stabiliser}  $\stab(W)$ of $W$ in $\Aut(\GG)$ is the set 
consisting 
of elements $\a$ such that  $w_i\a=w_i$,  for $1\le i\le k$, whereas the \emph{stabiliser up to conjugacy} $\stab_\sim(W)$ of $W$ is the set of 
elements $\a$ such that $w_i\a$ is conjugate to $w_i$, for $1\le i\le k$, 
(the stabiliser of $W$ as a tuple of conjugacy 
classes). 

\begin{theorem}\label{thm:stgen}
Let $x\in X$. Then $\St_x^\v=\stab(X\backslash [x])$ and $
\St_x^\v$ is generated by $\stab(X\backslash [x])\cap(\W_l\cup \W_s)$. In 
particular $\St^\v_{x,s}$ is generated by  
$\stab(X\backslash [x])\cap \W_s$ and $\St^\v_{x,l}$ is generated by  
$\stab(X\backslash [x])\cap \W_l=\W_x$.
\end{theorem}
To prove this we use the analogue of \cite[Corollary 4.5]{Day09} for 
$\stab(W)$ instead of $\stab_\sim(W)$. 

\begin{prop}[\textit{cf.} {\cite[Corollary 4.5]{Day09}}]\label{prop:stabtuple}
Let $W=(w_1,\ldots ,w_k)$ be a $k$-tuple of elements of $L$, where $k\ge 2$ and $\v(w_i)\neq \v(w_j)$, if $1\le i<j\le k$. Then  the 
subgroup $\stab(W)$ is generated by $(\W_l\cup \W_s)\cap \stab(W)$. 
\end{prop}
\begin{proof}
Let $\a\in \stab(W)$. From \cite[Corollary 4.5]{Day09}, $\a\in \la (\W_l\cup \W_s)\cap \stab_\sim(W)\ra$. In fact, from the proof 
of   \cite[Proposition C]{Day09} there is a factorisation, which is peak reduced with respect to $W$, 
\[\a=\phi_1\cdots \phi_r \s^{-1}\d_1\cdots \d_m,\]
where $\phi_i\in \W_s$, $\s\in \Aut(\G^{\pm})$ and $\d_i\in \W_l$, with $W\phi_i=W$; and so $|W\s^{-1}\d_1\cdots \d_i|_\sim=|W|_\sim$, 
for $1\le i\le m$. If $\d_i$ is of Type 1 then so is $\d_i^\s$ and if $\d_i$ is of Type 2 then so is $\d_i^\s$; and (from DR6) 
\[\s^{-1}\d_1\cdots \d_m=\d_1^\s\cdots \d_m^\s\s^{-1}.\]
Moreover, as elements of $\Aut(\G^{\pm 1})$ do not affect length, $|W|_\sim=|W\s^{-1}\d_1|_\sim=|W\d_1^\s|_\sim$ and similarly 
$|W\s^{-1}\d_1\cdots \d_i|_\sim=|W\d_1^\s\cdots \d_m^\s|_\sim$, for $1\le i\le m$. Hence we may replace the above factorisation
with $\a=\phi_1\cdots \phi_r \d_1^\s\cdots \d_m^\s\s^{-1}$, 
which is also peak reduced with respect to $W$. 

Continuing in this way we may move any of the $\d_i$ which are of Type 1 to the right hand end of the factorisation, until we have, 
after renaming, a peak reduced factorisation 
\[\a=\phi_1\cdots \phi_r \d_1\cdots \d_m\s^{-1},\]
satisfying $\phi_i\in \W_s\cap \stab(W)$, $\s\in \Aut(\G^{\pm})$, $\d_i\in \W_l$ of Type 2, and $|W\d_1\cdots \d_i|_\sim=|W|_\sim$, 
for $1\le i\le m$.

This means that  
$|W\d_1|_\sim=|W|_\sim$ and $|w_i\d_1|_\sim\ge 1$, with equality only if either $w_i\d_1=w_i$ or $w_i\d_1=w_i^a$, for some $a\in L$. It follows
that $w_i\d_1=w_i$ or $w_i^a$, $a\in L$, for $1\le i\le k$. That is, writing $\d_i=(A_i,a_i)$, with $A_i\cap \lk_L(a_i)=\nul$,  
and partitioning $A_i\backslash \{a_i\}$ 
 as the disjoint union $A_{i,0}\cup A_{i,1}$, where $A_{i,0}=A_{i,0}^{-1}$ and $A_{i,1}\cap A_{i,1}^{-1}=\nul$, as in Remark  
\ref{rem:whrem}.\ref{it:whrem2}, 
we have $\{w_1,\ldots,w_k\}\cap A_1\backslash \{a_1\} =\{w_1,\ldots,w_k\}\cap A_{1,0}$  and $\{w_1,\ldots,w_k\}\cap A_{1,1}^{\pm 1}=\nul$). 
Hence, $w_i$ is a minimal length representative of the conjugacy class of  $w_i\d_1$, for $1\le i\le k$.

Assume now that, for some $j\ge 1$, 
$w_i$ is a minimal length representative of the conjugacy class of  $w_i\d_1\cdots \d_j$, for $1\le i\le k$. Then
$w_i\d_{j+1}$ is a representative of the conjugacy class of  $w_i\d_1\cdots \d_j\d_{j+1}$ and, if it has length $2$ it
is conjugate to no shorter element. Therefore $w_i\d_{j+1}$ is again equal to $w_i$ or $w_{i}^{a_{j+1}}$, 
\[\{w_1,\ldots, w_k\}\cap A_j\backslash\{a_j\}\subseteq A_{j,0}\textrm{ and } \{w_1,\ldots, w_k\}\cap A_{j,1}^{\pm 1}=\nul,\]
 and 
$w_i$ is a representative of the conjugacy class of  $w_i\d_1\cdots \d_j\d_{j+1}$, for $1\le i\le k$. 

Therefore, there exist $g_i\in \GG$ such that $w_i\d_1\cdots \d_m=g_i^{-1}\circ w_i\circ g_i$. As $W=W\d_1\cdots \d_m\s^{-1}$, this
implies that $g_i^{-1}\circ w_i\circ g_i=w_i\s\in L$. Hence $g_i=1$ and $w_i\s=w_i$, for $1\le i\le k$. Thus 
$\s\in \Aut(\G^{\pm 1})\cap \stab(W)$ and $W=W\d_1\cdots \d_m$, where $\d_i\in \W_l$ is of Type 2, for $i=1,\ldots ,m$.

If $V=(v_1,\ldots, v_k)$ and $W=(w_1,\ldots, w_l)$ are tuples of elements (or conjugacy classes) of $\GG$, let 
$VW$ denote the concatenation $(v_1,\ldots, v_k,w_1,\ldots, w_l)$ of $V$ and $W$. Inductively define $W^n=W^{n-1}W$, for $n\ge 2$. 
Given $W=(w_1,\ldots, w_k)$ let $V_i=(w_i^2,w_iw_{i+1},\ldots ,w_iw_k)$, for $1\le i\le k$, let 
\[V=V_1\cdots V_k \textrm{ and } Z=W^{k(k+1)}V,\]
so $|Z|_\sim = k(k+1)^2$, and $Z$ is fixed point wise by $\a$, $\phi_i$, $\d_1\cdots \d_m$ and $\s$. 

Let $\a_1=\d_1\cdots \d_m$. Applying \cite[Lemma 3.18]{Day09}, we may choose a factorisation 
\[\a_1=\b_1\cdots \b_t,\]
which is peak reduced with respect to $Z$, where $\b_j\in \W_l$, for all $j$.  
Let $j$ be minimal such that there exists some $i$ with $w_i\b_1\cdots \b_j$ not equal to a conjugate of an element of $L$.
Then $|w_i\b_1\cdots \b_{s-1}|_\sim=1$, and $w_i\b_1\cdots \b_{j-1}=v_i$, for some $v_i\in L$, 
for all $i,s$ such that $1\le i\le k$ and $1\le s\le j$; while there exists $i$ such that $v_i\b_j=v_ia$ or $av_i$, for some 
$a \in L$. Assume that there are precisely $r$ elements of $i\in \{1,\ldots, k\}$ such that $v_i\b_j$ is not conjugate to an element
of $L$; so $|W \b_1\cdots \b_j|_\sim=k+r>k$. As each element of the tuple $V \b_1\cdots \b_j$ is conjugate to an element of length at least 
$1$, we have $|V \b_1\cdots \b_j|_\sim\ge k(k+1)/2$. Therefore
\[
|Z \b_1\cdots \b_j|_\sim\ge (k+r)\left(
k(k+1)
\right)
+\frac{k(k+1)}{2}>k(k+1)^2=|Z|_\sim.
\]
As $|Z\b_1\cdots \b_t|=|Z|$ and $\b_1\cdots \b_t$ is peak reduced with respect to $Z$, this cannot occur, there is no such $j$, and 
$w_i\b_1\cdots \b_j$ is conjugate to an element of $L$, for all $i,j$.

Next we shall move all $\b_j$'s of Type 1 to the right hand end of the factorisation of $\a_1$. 
We may assume no two consecutive $\b_i$'s are  of Type 1. Let  $j$ be 
minimal such that $\b_j$ is of Type 1 and assume that $j<t$. 
 Assume first that $\b_s$ is of Type 2, for all $s>j$. 
Let $Z'=Z\b_1\cdots \b_{j-1}$, so writing $\t=\b_j^{-1}$, 
\[Z\b_1\cdots \b_t=Z'\b_j\cdots \b_t=Z'\b_{j+1}^\t \cdots \b_t^\t\b_j.
\]
Also, for $j+1\le s\le t$, 
\[
|Z|_\sim=|Z'|_\sim=|Z'\b_j\b_{j+1}\cdots \b_s|_\sim=|Z'\t\b_{j+1}\cdots \b_s\t|_\sim=|Z'\b_{j+1}^\t\cdots \b_s^\t|_\sim,
\]
so 
\[\a_1=\b_1\cdots \b_{j-1}\b_{j+1}^\t \cdots \b_t^\t\b_j
\]
 is also peak reduced with respect to $Z$.
In the case where $\b_s$ is also of Type 1, for some $s$ such that $j<s\le t$ we set $\b'_s=\b_j\b_s$, so $\b'_s$ is of Type 1, and 
the same argument shows that the factorisation 
\[
\a_1= \b_1\cdots \b_{j-1}\b_{j+1}^\t \cdots \b_{s-1}^\t\b'_s\b_{s+1}\cdots \b_t
\]
 is  peak reduced with respect to $Z$. 
In all cases the new factorisation has fewer elements of 
Type 1 to the left of elements of Type 2 than the original factorisation.
Thus, continuing in this way we may assume that we have a factorisation 
$\a_1=\b_1\cdots \b_t\s_1$ which is peak reduced with respect to $Z$, such that $\b_i$ is of Type 2 in $\W_l$ and $\s_1$ is of Type 1. 
Moreover, from the above, for $1\le i\le k$ and $1\le j\le t$ we have  $w_i\b_1\cdots \b_j$ conjugate to an element of $L$; so 
we have $w_i\b_1\cdots \b_j$ conjugate to $w_i$, for all $i,j$.  
Also as above, since $Z\a_1=Z$ it follows that $w_i\s_1=w_i$ and $w_i\b_1\cdots \b_t=w_i$, for all $i$. 

Given a factorisation $\a_1=\b_1\cdots \b_t\s_1$ with the properties above, let $\b_j=(B_j,b_j)$, where 
$B_j\cap \st_L(b_j)=\{b_j\}$, for all $j$. We claim that, for all $j$, 
either 
\be
\item\label{it:wint1} $\{w_1,\ldots ,w_k\}\subseteq B_j\cup \st_L(b_j)$ or 
\item\label{it:wint2}  $\{w_1,\ldots ,w_k\}\subseteq (L-B_j)\cup \st_L(b_j)$. 
\ee
To prove this claim, assume that $j$ is minimal such that the claim does not hold. Note that, as $w_i\b_1\cdots\b_j$ is conjugate
to $w_i$, we have $w_i\in B_j\backslash\{b_j\}$ if and only if $w_i^{-1}\in B_j\backslash\{b_j\}$. Suppose that, for some $p,q$ we have 
$w_p\in B_j$, $w_p\neq b_j^{\pm 1}$ and $w_q\in L-(B_j\cup\st_L(b_j))$.
Then $[w_p,b_j]\neq 1$ and $w_p\b_j=\b_j^{-1}\circ w_p\circ \b_j$ while $[w_q,b_j]\neq 1$ and $w_q\b_j=w_q$. 
Hence $w_pw_q\b_j=b_j^{-1}\circ w_p\circ b_j\circ w_q$; a minimal length representative of its conjugacy class in $\GG$. 

By assumption the claim holds for $1,\ldots , j-1$, so there exists $g\in \GG$ such that $w_i\b_1\cdots \b_{j-1}=w_i^g$. 
Thus $w_pw_q\b_1\cdots \b_{j}$ is conjugate to $w_pw_q\b_j$ and therefore $|w_pw_q\b_1\cdots \b_{j}|_\sim\ge 4$. It follows
that  $|w_qw_p\b_1\cdots \b_{j}|_\sim\ge 4$ and either $w_pw_q$ or $w_qw_p$ occurs in $V$. For all other $w_rw_s$ in $V$ we
have $|w_rw_s\b_1\cdots \b_{j}|_\sim\ge 2$ and $|W\b_1\cdots \b_{j}|_\sim=|W|_\sim$. Hence 
\[
|Z\b_1\cdots \b_{j}|_\sim\ge |Z|_\sim +2,
\]
a contradiction. Therefore no such $j$ exists, and the claim holds.

Now let $j$ be minimal such that \ref{it:wint1} above holds. Define 
\[\b'_j=(L-B_j-\lk_L(b_j)+b_j^{-1},b_j^{-1}).\]
Then $\b_j'\in \W_l$,  
 $\b_j=\b'_j\g_{b_j}$, where $\g_{b_j}$ is conjugation by $b_j$, and  
\[
\b_1\cdots \b_t=\b_1\cdots \b_{j-1}\b'_j\g_{b_j}\b_{j+1}\cdots \b_t=\b_1\cdots \b_{j-1}\b'_j\b_{j+1}\cdots \b_t\g',
\] 
for some $\g'\in \Inn(\GG)$. As $\{w_1,\ldots ,w_k\}\subseteq B_j\cup \st(b_j)$, for the latter factorisation, $j$ satisfies condition
\ref{it:wint2} above. We have thereby reduced the number of indices $j$ for which \ref{it:wint1} holds.  Continuing this way
we may assume $\a_1=\b_1\cdots \b_t\g\s_1$, where $\b_j$ is of Type 2 and  satisfies \ref{it:wint2} above, for $j=1,\ldots ,t$, $\g\in\Inn(\GG)$ and 
$\s_1\in \W_l$ of Type 1, such that $W\s_1=W$. In this case $W=W\b_1\cdots \b_t=W\g$, so $\g$ is conjugation by $g\in \GG$ such that 
$g\in C_\GG(w_i)$, for $1\le i\le k$. Every element of $\Inn(\GG)$ is a product of elements $\W_l$ and it follows that $\g$ is a product
of elements of $\W_l\cap \stab(W)$. Hence we have a factorisation of $\a_1$ as a product of elements of $\stab(W)\cap \W_l$.  
As $\phi_i\in \W_s\cap \stab(W)$ and $\s\in \Aut(\G^{\pm 1})\cap \stab(W)$, it follows that $\a$ belongs to the subgroup generated by
$(\W_l\cup \W_s)\cap \stab(W)$, as required.
\end{proof}
\begin{proof}[Proof of Theorem \ref{thm:stgen}]
The final statement follows from the first, in view of Theorem \ref{thm:stxdecomp} and Definition \ref{defn:stx1gen}.
By definition, $\St^\v_x\subseteq \stab(X\backslash [x])$. 

For the opposite inclusion, first consider the case $|X\bs[x]|\le 1$. If $|X\bs [x]|=0$, then
$\stab(X\bs [x])=\Aut(\GG)=\St^\v_x$ and the Theorem follows from the results of Laurence and Servatius.
If $|X\bs[x]|= 1$ then $\GG(\G)= \la a \ra\times \GG([x])$, where $\la a \ra$ is infinite cyclic
generated by  the element $a$ of $X\bs [x]$ and $\GG([x])$ is the  free group on $[x]$. 
If $\phi \in \stab(X\bs [x])$ then $a\phi=a$ and, for all $x\in X$,  $x\phi=w_xa^{n_x}$, where
$w_x\in \GG([x])$ and $n_x\in \ZZ$. Hence $\phi \in \St_x^\v$. Let $\t=\prod_{x\in X}\tr_{x,a}^{n_x}$ and let $\theta=\t^{-1}\phi$.
Then $a\theta=a$ and $x\theta=(xa^{-n_x})\phi=w_x\in \GG([x])$, for all $x\in X$. Therefore $\theta$ restricts to
an element of $\Aut(\GG([x]))$; which can be written as a product of Whitehead automorphisms of $\GG([x])$, and these
may all be regarded as Whitehead automorphisms of $\GG$, fixing $a$, and necessarily in $\W_l$, hence in $\W_l\cap \stab(X\bs [x])$.
Moreover $\t$ is a product of elements of $\W_s$, which are also in $\stab(X\bs [x])$. Therefore $\phi$ is in $\St_{x}^\v$ and
is a product of elements of $(\W_l\cup \W_s)\cap \stab(X\bs [x])$, as required.

Now consider the case $|X\bs[x]|\ge 2$.  From Proposition \ref{prop:stabtuple}, 
it suffices to show that every element of $\stab(X\backslash [x])\cap (\W_l\cup \W_s))$ 
belongs to $\St^\v_x$. 
If $\s$ is a Type 1 element of $\stab(X\backslash [x])\cap (\W_l\cup \W_s))$ then 
$\s$ permutes elements of $[x]_L$ and fixes all other elements of $L$, so belongs 
to $\St^\v_x$. If $(A,a)$ is of Type 2 in $\stab(X\backslash [x])\cap (\W_l\cup \W_s))$
then, by definition of $\stab(X\backslash [x])$, $A-a\subseteq [x]_L$. It remains to
show that $a\in \ad(x)^{\pm 1}$. If $a\notin \ad(x)^{\pm 1}$ then there is $y\in \lk(x)$ such
that $a\notin \st_L(y)$. In this case $\lk(x)\nsubseteq \st(a)$ and so 
$A-a$ consists of the  vertices of a union of connected components of $\G\backslash \st(a)$, and their inverses.  
Since $A-a$ contains some element of $[x]_L$ (as we assume $(A,a)$ is non-trivial) this means that $y\in A$, a contradiction. Therefore $(A,a)\in  \St^\v_x$. 
\end{proof}
\subsection{Peak reduction for $\St_x^\v$}\label{sec:peakxv}
First note that the elements of $\W_{x}$ are all, by definition, long range and in fact 
$\W_x=\St_x^\v\cap \W_l$. We shall  need the following Lemma in the proof of Lemma \ref{lem:Daypl}. 
\begin{lemma}\label{lem:321*}
Let $\a=(A,a)$ and $\b=(B,b)$ be Type 2 elements in $\W_l$ and 
 let  $C$ be a $k$-tuple of conjugacy classes of $\GG$, such that $A\subseteq B$ and $\a^{-1}\b$ is a peak for $C$.  
In this case 
\be[label=(\roman*)]
\item\label{it:321*1} if  $a^{-1}\in B$ then 
$|C\b|_\sim <|C\a^{-1}|_\sim$ and 
\item\label{it:321*2} if 
$b\notin A$ 
then $|C(B-A+a,b)|_\sim<|C\a^{-1}|_\sim$.  
\ee
\end{lemma}
\begin{proof}
Let $\b^*=(L-B-\lk_L(b),b^{-1})$. Then  $\b=\b^*\g_b$, where $\g_b=(L-b^{-1},b)$ is conjugation by $b$. 
As $A\subseteq B$ we have $A\cap (L-B-\lk_L(b))=\nul$.
\be[label=(\roman*)]
\item
 As $a^{-1}\in B$ we have $a^{-1}\notin L-B-\lk_L(b)$. 
Also $|C\a^{-1}\b|_\sim=|C\a^{-1}\b^*\g_\b|_\sim=|C\a^{-1}\b^*|_\sim$, as $\g_g$ is inner, and 
  $\a^{-1}\b$ is a peak for $C$, so $\a^{-1}\b^*$ is a peak for $C$. From \cite[Lemma 3.21]{Day09}, we have
$|C\b^*|_\sim<|C\a^{-1}|$, and again $|C\b|_\sim=|C\b^*|_\sim$. 
\item 
As before, since $\a^{-1}\b$ is a peak for $C$, so is $\a^{-1}\b^*$, and so $(\b^*)^{-1}\a$ is a peak for
$C^*=C\a^{-1}\b^*$. As $b\notin A$, from \cite[Lemma 3.21]{Day09}, we have
\[|C\a^{-1}\b\a|_\sim=|C\a^{-1}\b^*\a|_\sim<|C\a^{-1}\b^*(\b^*)^{-1}|_\sim=|C\a^{-1}|_\sim.\]
Write $B^*=L-B-\lk_L(b)$. From DR4, we have $\a^{-1}\b^*\a=(A+B^*-a,b^{-1})$ so 
\[
\a^{-1}\b\a=\a^{-1}\b^*\g_b\a=\a^{-1}\b^*\a\g_b=(A+B^*-a,b^{-1})\g_b=(B-A+a,b).\]
\ee
\end{proof}
\begin{lemma}[\textit{cf.} {\cite[Lemma 3.18]{Day09}}]\label{lem:Daypl}
 Let $C$ be a $k$-tuple of conjugacy classes of $\GG$ and $\a,\b\in \W_x$ such that $\a^{-1}\b$ is a peak with respect to $C$. Then
\be[label=PL\arabic*.,ref=PL\arabic*]
\item\label{it:Dpl1}
there is a peak lowering  $\a^{-1}\b=\d_1\cdots \d_s$ with respect to $C$, in $\W_x$, and 
\item\label{it:Dpl2}
the relation $\a^{-1}\b=\d_1\cdots \d_s$ follows from \ref{it:DR1}--\ref{it:DR6}, \ref{it:DR3*} and \ref{it:DR4*} above. 
\ee
\end{lemma} 
\begin{proof}
We may assume that $C=(c_1,\ldots c_k)$ where $c_i$ is a minimal representative of its conjugacy class; so
$c_i$ and all its cyclic permutations are minimal words.  From \cite[Lemma 3.18]{Day09} there is a peak lowering of 
$\a^{-1}\b$ in $\W_l$. We work through Cases 1 to 4 of the proof of Lemma 3.18 in \cite{Day09} to show that in the 
case in hand we may find such a factorisation satisfying \ref{it:Dpl1} and \ref{it:Dpl2}. Cases 1 to 3 go through
in the same way as they do in \cite{Day09}. To cope with Case 4, without using automorphisms from outside $\W_x$, 
we extend the treatment of Case 3, following McCool \cite{Mccool74}. 
(Note that in \cite{Day09} automorphisms act on the left, whereas here automorphisms act on the right.)
\\[1em]
\textbf{Case 1.}  The case where $\a\in \Aut(\G^{\pm})$: that is $\a$ is in $\W_x$ and of Type 1. The peak 
lowering factorisation in \cite{Day09} is 
\begin{equation}\label{eq:pl1}
\a^{-1}\b=\b'\a^{-1},
\end{equation} where $\b=(B,b)$ and $\b'=(B\a,b\a)$. As 
$\a$ and $\b$ are in $\W_x$, we have $b\a=b$ or $b\a\in [x]_L$ and $(B-b)\a\subseteq [x]_L$; so $\b'\in \W_x$. 
Moreover, relation \eqref{eq:pl1} follows from \ref{it:DR6}. 

From now on we assume $\a$ and $\b$ are of Type 2, $\a=(A,a)$ and $\b=(B,b)$. Moreover, for elements $a,b\in L$, with
$a\neq b^{\pm 1}$ we define $\s_{a,b}$ to be the element of $\Aut(\G^{\pm})$ which fixes all elements of $L$ not
 equal to $a^{\pm 1}$ or $b^{\pm 1}$, maps $a$ to $b^{-1}$ and $b$ to $a$ (as in \ref{it:DR5}).\\[1em]
\textbf{Case 2.} The case $a\in \lk_L(b)$. 
 From \ref{it:DR3}, that $\a^{-1}\b=\b\a^{-1}$, so, as in \cite{Day09}, both \ref{it:Dpl1} and \ref{it:Dpl2} hold.\\[1em]
\textbf{Case 3.} 
The conditions of this case are that $A\cap B=\nul$ and $a\notin \lk_L(b)$. The case is broken (in \cite{Day09}) into
three sub-cases, a, b and c. As noted in \cite{Day09} these three sub-cases exhausts all
possibilities in Case 3.  \\[1em]
\textbf{Sub-case 3a.} 
In this sub-case $\v(a)=\v(b)$. From \cite{Day09} we have a peak lowering
\[\a^{-1}\b=(A+B+a^{-1},a^{-1})=(A-a+a^{-1},a^{-1})(B,a^{-1}),\]
which follows from \ref{it:DR1} and \ref{it:DR2}, since $(A+B+a^{-1},a^{-1})\in \W_x$.\\[1em]
\textbf{Sub-case 3b.} 
In this sub-case $a^{-1}\notin B$. If $b^{-1}\notin A$ then $b^{\pm 1}\notin A$ and so from \ref{it:DR3},
we have $a^{-1}\b=\b\a^{-1}$. As in \cite{Day09}, this factorisation is peak lowering. 

If $b^{-1}\in A$ then, 
from \ref{it:DR4},  we have 
\[\a^{-1}\b=\b(A+B-a-b+a^{-1},a^{-1}),\]
and, as in \cite{Day09}, this factorisation is peak lowering.\\[1em]
\textbf{Sub-case 3c.} 
In this sub-case $\v(a)\neq \v(b)$, $a^{-1}\in B$ and $b^{-1}\in A$. 
 The conditions that $a^{-1}\in B$ and $\v(a)\neq \v(b)$ imply that $a^{-1}\in B-b\subseteq [x]_L$. 
Similarly, $b^{-1}\in [x]_L$, so $a^{\pm 1},b^{\pm 1}\in [x]_L$. It follows that $\a'=(A,b^{-1})$, 
$\b'=(B,a^{-1})$ and $(B-a^{-1}+a-b+b^{-1},a)$ are in $\W_x$, as is the element $\s_{a,b}$ of  $\Aut(\G^{\pm})$. From \ref{it:DR5},
\[(\b')^{-1}\b=\s_{a,b}(B-a^{-1}+a-b+b^{-1},a)\]
and from \ref{it:DR2},
\[\a^{-1}\b'=(A+B-a,a^{-1})\in \W_x.\]
Hence 
\[\a^{-1}\b= \a^{-1}\b'(\b')^{-1}\b = (A+B-a,a^{-1})\s_{a,b}(B-a^{-1}+a-b+b^{-1},a)\]
and as in \cite{Day09} this factorisation is peak lowering; so \ref{it:Dpl1} and \ref{it:Dpl2} hold in this case.

This concludes Case 3. Some instances of peak lowering in Case 4, in \cite{Day09}, require $\a$ to be replaced by $\a^*=(L-A-\lk_L(a),a^{-1})$,
and if $\a\in\W_x$ then  $\a^*$ is not. To avoid this replacement we consider the analogue of Case 3 in which 
we assume $A\subseteq B$ instead of $A\cap B=\nul$. (In the usual treatment of peak lowering these cases follow, 
after switching
$\a$ and $\a^*$, or making a similar switch for $\b$.)  \\[1em]
\textbf{Case 3$^*$.} Assume $A\subseteq B$ and $a\notin \lk_L(b)$. 
We break the case into three sub-cases.\\[1em]
\textbf{Sub-case  3$^*$a.}  
In this sub-case $\v(a)=\v(b)$. As $A\subseteq B$ this implies $a=b$. Then
\[\a^{-1}\b=(A-a+a^{-1},a^{-1})(A,a)(B-A+a,a)=(B-A+a,a),\]
is a peak lowering factorisation and, as $(B-A+a,a)\in\W_x$, this relation follows from \ref{it:DR2} and 
\ref{it:DR1}, so \ref{it:Dpl1} and \ref{it:Dpl2} hold.\\[1em]
\textbf{Sub-case 3$^*$b.}
Assume $a^{-1}\in B$. As $a\in B$ this implies that $\v(a)\neq \v(b)$. If $b\notin A$ then, from 
\ref{it:DR3*}, $\a^{-1}\b=\b\a^{-1}$. 

If $b\in A$ then $(B-A+b^{-1},a^{-1})\in \W_x$ and, from \ref{it:DR4*}, $\a^{-1}\b=\b(B-A+b^{-1},a^{-1})^{-1}$.
In both cases it follows from Lemma \ref{lem:321*} that these factorisations are peak lowering. \\[1em]
\textbf{Sub-case  3$^*$c}
In this sub-case $\v(a)\neq \v(b)$ and $a^{-1}\notin B$. 
 If $b\notin A$ then $(B-A+a,b)\in \W_x$ so from \ref{it:DR4},
\[
\b=(A-a+a^{-1},a^{-1})^{-1}(B-A+a,b)(A-a+a^{-1},a^{-1}).
\]
This gives a factorisation $\a^{-1}\b=(B-A+a,b)\a^{-1}$, which
Lemma \ref{lem:321*} \ref{it:321*2} implies is peak lowering. 

If $b\in A$ then $a,b\in [x]_L$, so $a\sim b$ and $\a_b=(A,b)$ and 
$\b_a=(B,a)$ are defined and in $\W_x$. Also, as in \cite{Day09}, 
(using an adjacency counter argument and \cite[Lemma 3.17]{Day09})
either $\a_b$ or $\b_a$ reduces $|C\a^{-1}|_\sim$. 
 Assume first that $|C\a^{-1}\b_a|_\sim<|C\a^{-1}|_\sim$. 
We have, from  \ref{it:DR1} and \ref{it:DR2}, that  $\a^{-1}\b_a=(B-A+a,a)$ and so 
\begin{align*}
\a^{-1}\b&= (B-A+a,a)\b_a^{-1}\b\\
&=(B-A+a,a)\s_{a^{-1},b}(B+a^{-1}-a+b^{-1}-b,a^{-1}),\\
\end{align*} 
using \ref{it:DR5}. 
As $\s_{a^{-1},b}$ preserves lengths and $|C(B-A+a,a)|_\sim<|C\a^{-1}|_\sim$, this
factorisation is peak lowering. 

On the other hand, if $|C\a^{-1}\a_b|_\sim<|C\a^{-1}|_\sim$ then,
from \ref{it:DR2} and\ref{it:DR6}, 
\[
\b=\a_b(B-A+b,b) \textrm{ and }\s_{a,b}^{-1}\a_b\s_{a,b}=(A-b+b^{-1},a),
\]
 so 
\begin{align*}
\a^{-1}\a_b&=\a^{-1}\s_{a,b}(A-b+b^{-1},a)\s_{a,b}^{-1}\\
&=(A-a+a^{-1},b)\s_{a,b}^{-1},
\end{align*}
using \ref{it:DR5} 
and, since $|C(A-a+a^{-1},b)\s_{a,b}^{-1}|_\sim<|C\a^{-1}|_\sim$, the factorisation 
\[\a^{-1}\b=(A-a+a^{-1},b)\s_{a,b}^{-1}(B-A+b,b)\]
is peak lowering. 
~\\[1em]
\textbf{Case 4.} In the light of Cases 1 to 3 and 3$^*$ above, and since we may interchange $\a$ and $\b$, we may now 
assume $a\notin \lk_L(b)$ and $A\cap B$, $A\cap B^c$  and $A^c\cap B$ are all non-empty.\\[1em]
\textbf{Sub-case  4a.}   
In this case we assume $\lk_L(a)= \lk_L(b)$ and $\v(a)\neq \v(b)$; so $a\sim b$. As in \cite{Mccool74} (using the 
form of adjacency counting defined in \cite{Day09}) 
after interchanging $\a$ and $\b$ if necessary, we may assume one 
of $\a_1=(A\cap B,a)$, $\a_2=(A\cap B^c,a)$, $\a_3=(A^c\cap B,a^{-1})$ or $\a_4'=(A^c\cap B^c,a^{-1})$ is a well-defined 
Whitehead automorphism and 
 reduces 
$|C\a^{-1}|_\sim$. If $\a_4'$ is defined and reduces $|C\a^{-1}|$, 
then by composing with the inner
automorphism $\g_a$, we see that $\a_4=(A\cup B,a)$ is also defined and also reduces $|C\a^{-1}|$. 
In addition, if it is defined,  
$\a_i\in \W_x$, for $i=1,\ldots ,4$. 
(In fact $\a_1$ is defined if $a\in B$, $\a_2$ if $a\notin B$, $\a_3$ if $a^{-1}\in B$ and $\a_4$ if $a^{-1}\notin B$.) 
Now, for $i=1,\ldots ,4$, define 
$\hat\a_i=\a^{-1}\a_i$. Using \ref{it:DR1} and \ref{it:DR2} we have, when the map in question is defined,  
\begin{align*}
\hat\a_1&=(A\cap B^c+a,a)^{-1},\\
\hat\a_2&=(A\cap B+a,a)^{-1},\\
\hat\a_3&=(A\cup B-a^{-1},a)^{-1}\textrm{ and}\\
\hat\a_4&=(A^c\cap B+a,a).
\end{align*}

Assume then  $1\le i\le 4$ and that $\a_i$ is defined and shortens $|C\a^{-1}|$. Then  $|C\hat\a_i|_\sim=|C\a^{-1}\a_i|_\sim<|C\a^{-1}|$; 
so $\a_i^{-1}\b$ is a peak with respect to $C\hat\a_i$.  
If $i=1,2$ or $4$ then Case 3$^*$ gives  a peak-lowering of $\a_i^{-1}\b$, with respect to $C\hat\a_i$, in $\W_x$. 
If $i=3$, then Case 3 gives  a peak-lowering of $\a_i^{-1}\b$, with respect to $C\hat\a_i$, in $\W_x$. 
In all cases we have a peak lowering factorisation  
\[\a_i^{-1}\b=\d_1\ldots \d_k,\]
with $\d_i\in \W_x$. Therefore, as   $\hat \a_i\in \W_x$, 
\[\a^{-1}\b=\hat\a_i\d_1\ldots \d_k,\]
is a peak-lowering factorisation  of $\a^{-1}\b$, in $\W_x$. 
Moreover this factorisation follows from the relations R$_x$. \\[1em]
\textbf{Sub-case  4b.}   
In this case we assume that $\v(a)=\v(b)$ or $\lk_L(a)\neq\lk_L(b)$.
We break this sub-case into two further sub-cases:  either $a\in B$ or $b\in A$; or  $a\notin B$ and   $b\notin A$. 
\be[label=(\roman*)]
\item If $b\in A$ but $a\notin B$ then interchanging $\a$ and $\b$ we obtain $a\in B$. Hence we may 
 assume that $a\in B$. In this case either $a=b$  or $\v(a)\neq \v(b)$.  If $a=b$ then  $a^{-1}=b^{-1}\notin A\cup B$. 
If  $\v(a)\neq \v(b)$ then $a\in B\backslash\{b\}\subseteq [x]_L$ and, as $\lk_L(a)\neq\lk_L(b)$, we have $b\notin [x]_L$, 
so $b^{\pm 1}\notin A$. In both cases $b^{-1}\notin A$ and $a\notin B^*$. If $A\cap B^*=\nul$ then $A\subseteq B\cup \lk_L(b)$ 
and, as $\a,\b\in \W_x$, we have $A\bs\{a\}\subseteq [x]_L$ and $b\in [x]_L\cup \ad_{\out,L}(x)$, 
so $A\bs\{a\}\cap \lk_L(b)=\nul$. As $a\in B$ this implies $A\subseteq B$, a contradiction. Hence, 
$A\cap B^*\neq \nul$ and, as $\a^{-1}\b$ is a peak for $C$ so is 
$\a^{-1}\b^*$. As in \cite[Sub-case 4b]{Day09} both $(A^*\cap B^*,b^{-1})$ and $(A\cap (B^*)^*,a)=(A\cap B,a)$ are defined, 
and one or other 
reduces the conjugacy length of $C\a^{-1}$. If   $(A\cap B,a)$ shortens $C\a^{-1}$ then, as $(A\cap B,a)\in \W_x$, we 
may construct a peak lowering as in the case when $i=2$ of Sub-case 4a above; via elements of $\W_x$ and following from the 
relations R$_x$. If $(A^*\cap B^*,b^{-1})$ shortens the conjugacy length of $C\a^{-1}$ then so does $(A\cup (B^*)^*,b)=(A\cup B,b)$. 
In this case, after interchanging $\a$ and $\b$ we construct a peak lowering, with the required properties, as in the case when 
$i=4$ of Sub-case 4a above.
\item
If $a\notin B$ and  $b\notin A$ then we are in the same situation as Sub-case 4b in \cite{Day09}. 
In this case both $(A\cap B^*,a)$ and $(A^*\cap B,b)$ are defined, necessarily in $\W_x$, and one or other
of them reduces the conjugacy length of $C\a^{-1}$. If this conjugacy length is reduced by $(A\cap B^*,a)$ then 
we may construct a peak lowering, with the required properties, as in Sub-case 4a above, where $i=2$. For the remaining
case we first interchange $\a$ and $\b$ and then proceed as before. 
\ee
\end{proof}
\subsection{Proof of Theorem \ref{thm:stxvf}}\label{sec:proofstxvf}
Let $P$ be the group with presentation $\la \W_x\,|\,R_x\ra$. The canonical map from $P$ to $\St^\v_{x,l}$, taking $\a\in \W_x$ to 
its realisation as an automorphism of $\GG$, induces a surjective homomorphism, in view of Theorem \ref{thm:stgen} and the fact that
all the relations of $R_x$ hold in $\St^\v_{x,l}$.
 It remains to show that this homomorphism is also injective. 

For the duration of this section $C_2$ denotes a fixed tuple of words of $\GG$ of length $2$, such that 
$C_2$  contains precisely one representative of each conjugacy class of $\GG([x])$ of length $2$. Note that, if $y, z\in [x]_L$, with
$y\neq z^{\pm 1}$ it follows that either $yz$ or $zy$ is in $C_2$.   
 
\begin{lemma}\label{lem:nored}
If $\b\in\W_x$ and $|C_2\b|_\sim\le |C_2|_\sim$ then either 
\be[label=(\roman*)]
\item\label{it:nored1} $\b$ is of Type 1, or 
\item\label{it:nored2}  $\b=(B,b)$, where $b\in [x]_L$ and  $B=[x]_L-b^{-1}$, or 
\item\label{it:nored3}  $\b=(B,b)$, where $b\in \ad_{\out,L}(x)$ and  $B=[x]_L$.
\ee
In  \ref{it:nored1}  $|C_2\b|=|C_2|$ and in all cases $|C_2\b|_\sim=|C_2|_\sim$.
\end{lemma}
\begin{proof}
  In the case where $|[x]|=1$, without loss of generality we may assume that $C_2=(x^2,x^{-2})$. 
  If $\b$ does not map $x$ to $x^{\pm 1}$ or to a conjugate of $x$ then evidently $|C_2\b|_\sim>|C_2|_\sim$. Therefore the result holds in this case.

  Assume then that $|[x]|\ge 2$. 
If $\b$ is of Type 1, then the claims of the Lemma follow from the definition of $\W_x$. 
Assume then that $\b$ is not of Type 1, so $\b=(B,b)$, where $B-b\subseteq [x]_L$. 
If $b\in [x]_L$ then $\b|_{\GG([x])}$ is an automorphism of the free group $\GG([x])$. From, 
for example, \cite{Day09}[Theorem 5.2], in this case the restriction of $\b$ is  an
inner automorphism of this group. Hence $B=[x]_L-b^{-1}$, as claimed. 

On the other hand, if $b\notin[x]_L$ then $b\in \ad_{\out,L}(x)$. In this case, if $y$ and $z$ are elements of $[x]_L$, with 
$y\neq z^{\pm 1}$, $y\in B$ and $z\notin B$ then $yz\b=ybz$ or $yz\b=b^{-1}ybz$, in both cases a cyclically minimal word of length at least $3$. 
We may assume  
$yz\in C_2$, and as $b$ is not in $[x]_L$, no conjugacy class in $C_2$ has its length reduced by $\b$. Hence
$|C_2\b|_\sim>|C_2|_\sim$, a contradiction. It follows, since the identity map is of Type 1, that $B=[x]_L$.
In both cases \ref{it:nored2} and \ref{it:nored3} the map $\b$ acts by conjugation on $[x]_L$, so the final 
statement of the Lemma holds. 
\end{proof}


We shall call elements of $\W_x$ of the form occurring in \ref{it:nored1}, \ref{it:nored2} and \ref{it:nored3} generators of \emph{Type}
$1_x$, $2a_{x}$ and $2b_x$, respectively. 
Now
\begin{itemize}
\item let $\a\in \FF(\W_x)$, say $\a=\phi_1\cdots \phi_n$, where $\phi_j\in \W_x \,(=\W_x^{-1})$, and this word is reduced.
\end{itemize}
We may assume that the length of the word $\phi_1\ldots \phi_n$ cannot be reduced by application of relations \ref{it:DR1} or \ref{it:DR7}.  
\begin{lemma}\label{lem:facta}
Assume $\phi_j$ is of Type $1_x$, $2a_{x}$ or $2b_x$, for $j=1,\ldots ,n$, and $\phi_1\cdots \phi_n$ is peak reduced with respect to $C_2$. Then
there exist elements $\a_i,\b_i$ and $\s$ of $\W_x$ such that $\a_i$ is of type $2a_x$, $\b_i$ is of Type $2b_x$, $\s$ is of Type $1_x$ and 
\be[label=(\roman*)]
\item in the group $P$ the element $\a$ is equal to $\a_1\cdots \a_r\b_1\cdots \b_s \s$ and  
\item in $\St^\v_{x,l}$ the factorisation  $\a=\a_1\cdots \a_r\b_1\cdots \b_s \s$ is peak reduced with respect to $C_2$.
\ee
\end{lemma}
\begin{proof}
Let $i$ be maximal such that $\phi_i$ is of Type $2a_x$ or $2b_x$ and $\phi_{i-1}$ is of Type $1_x$. Then $\phi_i=(A,a)$. Let $\phi'_i=(A\phi_{i-1}^{-1},a\phi_{i-1}^{-1})$. 
From the definitions, $\phi'_i$ is of the same Type as $\phi_i$
 and 
relations \ref{it:DR6} imply that $\phi_{i-1}\phi_i=\phi'_i\phi_{i-1}$ in $P$. 

Let $C_{2,j}=C_2\phi_1\cdots \phi_j$, for $j=1,\ldots, n$, and $C_{2,0}=C_2$. As $\phi_1\cdots \phi_n$ is peak reduced with respect to $C_2$ it follows,
from Lemma \ref{lem:nored}, that $|C_{2,j}|_\sim=|C_2|_\sim$, for all $j$. As $\phi_{i-1}$ does not alter lengths, and $\phi_{i}$ and 
$\phi'_i$ preserve conjugacy lengths of elements of $\GG([x])$, we have \[|C_{2,i-2}|_\sim=|C_{2,i-2}\phi'_{i}|_\sim= |C_{2,i-2}\phi'_{i}\phi_{i-1}|_\sim,\]
from which it follows (as $\phi'_i\phi_{i-1}=\phi_{i-1}\phi_i$ in $\Aut(\GG)$) that the factorisation $\a=\phi_1\cdots\phi'_i\phi_{i-1}\cdots \phi_n$ 
is peak reduced with respect to $C_2$. Continuing this way we may move all the $\phi_i$ of Type $1_x$ to the right hand side, and use relations 
\ref{it:DR7},  to obtain 
$\a=\g_1\cdots \g_m\s$ in $P$, where $\g_i$ is of Type $2a_x$ or $2b_x$, $\s$ is of Type $1_x$ and the factorisation is peak reduced with 
respect to $C_2$. 

Now let $i$ be maximal such that $\g_i$ is of type $2a_x$ and $\g_{i-1}$ is of Type $2b_x$. By definition of Types $2a_x$ and $2b_x$, relations
\ref{it:DR3*} imply that $\g_{i-1}\g_i=\g_i\g_{i-1}$. As in the previous case, since $\g_1\cdots \g_m\s$  is peak reduced with respect to $C_2$, so
is $\g_1\cdots\g_i\g_{i-1} \cdots \g_m\s$. Continuing this way gives the required result. 
\end{proof}
\begin{lemma}\label{lem:xcon}
Let $y$ be a word of length $2$ in $\GG([x])$ and let $\a=\a_1\cdots \a_r\b_1\cdots \b_s \s$ be a factorisation of $\a$, 
satisfying  conditions of the conclusion of Lemma \ref{lem:facta}. 
Then $y\a=z^{uv}$, where $z$ is a word of length $2$ in $\GG([x])$, $u\in \GG([x])$ and $v\in \GG(\ad_\out(x))$. 
\end{lemma}
\begin{proof}
First assume that $\s=1$. 
If $r+s=0$ there is nothing to prove. 
Assume next that $s>0$ and inductively that $\b_s=(B,b)$, 
$y\a_1\cdots \a_r\b_1\cdots \b_{s-1}=z^{uv}$, for some  word $z$ of length $2$ in $\GG([x])$, $u\in \GG([x])$ and $v\in \GG(\ad_\out(x))$.
Then $y\a=(z^{uv})\b_s=z^{ubv}=z^{uv'}$, where $v'\in \GG(\ad_\out(x))$, so $y\a$  is of the required form. If $s=0$ and $r>0$ then again
$y\a=(z^{uv})\a_r=z^{u'v}$, where $u'=ua$, so $y\a$ is of the required form. Finally, if $\s\neq 1$ then, since $\s$ fixes $\ad_\out(x)$ point-wise
and permutes the elements of $[x]_L$, the result follows.
\end{proof}
\begin{lemma}\label{lem:aone}
Let $\a=\a_1\cdots \a_r\b_1\cdots \b_s \s$ be a factorisation of $\a$, 
as in the conclusion of Lemma \ref{lem:facta}. Assume $\a=1$ in $\St^\v_{x,l}$. Then $\a=1$ in $P$. 
\end{lemma}
\begin{proof}
Let $C_2=(y_1,\ldots, y_k)$. Then, from Lemma \ref{lem:xcon}, we have $C_2\a=(z_1^{u_1v_1},\ldots, z_k^{u_kv_k})$, where 
$z_i$ is of length $2$ in $\GG([x])$, $u_i\in \GG([x])$ and $v_i\in \GG(\ad_\out(x))$. As $\a=1$ in $\St^\v_{x,l}$, we have $C_2\a=C_2$ so
$z_i^{u_iv_i}=y_i$, for $i=1,\ldots, k$. No letter of $\ad_\out(x)$ commutes with any element of $[x]$, so this implies that $v_i=1$, for 
all $i$. Let $\b_j=([x]_L,b_j)$, for $j=1,\ldots ,s$; so $v_i=b_s\cdots b_1=1$, for all $i$.  Therefore there exist $1\le p<q\le s$ such that 
$b_p=b_q^{-1}$ and $b_j\in \lk_L(b_p)$, for $p<j<q$.
From \ref{it:DR3} \ref{it:DR3b} and \ref{it:DR1} it follows that, in $P$,  
\[
\b_p\b_{p+1}\cdots \b_{q-1}\b_q=\b_{p+1}\cdots \b_{q-1}\b_p\b_q=\b_{p+1}\cdots \b_{q-1},\]
so 
\[\b_1\cdots \b_s
=\b_1\cdots \b_{p-1}\b_{p+1}\cdots \b_{q-1}\b_{q+1}\cdots \b_s.\]
Continuing this process we eventually obtain $\a=\a_1\cdots \a_r\s$ in $P$. Then $y_i\a=z_i^{u_i}=y_i$, where $z_i$ and $u_i$ are as before,
for $i=1,\ldots ,k$.

If $[x]=\{x\}$, then there are no automorphisms of Type $2a_x$ and so $\a=\s$ is of Type $1_x$. As $\a=1$ in $\St^\v_{x,l}$, this implies that 
$\s$ is the identity permutation of $[x]_L$, so $\s=1$ in $P$, as required. Assume then that $|[x]|\ge 2$. For $1\le j\le r$ let 
$\a_j=([x]_L-a^{-1}_j,a_j)$, where $a_j\in [x]_L$. Then, for $1\le i\le k$ we have $y_i=y_i\a=y_i^{a_r\cdots a_1}\s$. In particular, if $x_1$ and $x_2$ are
distinct elements of $[x]$ then $x_1^2$ and $x_2^2$ both appear in $C_2$, 
\[x_1^2\s^{-1}=(x_1^2)^{a_r\cdots a_1}\textrm{ and }  x_2^2\s^{-1}=(x_2^2)^{a_r\cdots a_1}.\]
This occurs only if $x_i\s=x_i$, for $i=1,2$. Furthermore, 
from  \cite{Baudisch77}, $x_i^2=(x_i^2)^{a_r\cdots a_1}$ only if $(a_r\cdots a_1)\in C_\GG(x_i)$, for $i=1,2$.
As $(a_r\cdots a_1)\in\GG([x])$ and $\GG([x])\cap C_\GG(x_1)\cap   C_\GG(x_2)=\{1\}$, we have 
$a_r\cdots a_1=1$ in $\GG([x])$. As $\GG([x])$ is free there must therefore exist $j$ such that $a_{j-1}=a_j^{-1}$, so $\a_{j-1}\a_j=1$ in $P$, using 
\ref{it:DR1}. Continuing this process we again obtain $\a=\s$ and as $\a=1$ in $\St^\v_{x,l}$ we now have $\a=1$ in $P$. 
\end{proof}
\begin{proof}[Proof of Theorem \ref{thm:stxvf}]
As observed above it is necessary only to show that the canonical homomorphism from $P$ to $\St^\v_{x,l}$ is injective. 
Let $\a\in \FF(\W_x)$ and assume that $\a=1$ in $\St^\v_{x,l}$. Write $\a=\phi_1\cdots \phi_{n}$, where $\phi_i\in \W_x$ and 
define $C_{2,j}=C_2\phi_1\cdots \phi_j$, for $1\le j\le n$, and $C_{2,0}=C_2$. If $\phi_1\cdots \phi_{n}$ is not peak reduced with 
respect to $C_2$ then we say that $\phi_j\phi_{j+1}$ is a peak of \emph{height} $m$ (for $C_2$) if $\phi_{j}\phi_{j+1}$ is a peak 
for $C_{2,j-1}$ and $|C_{2,j}|_\sim=m$. Let $m$ be the maximum of the heights of peaks for $C_2$ and let $p$ be minimal such that 
$\phi_p\phi_{p+1}$ is a peak of height $m$.  This implies that $|C_{2,p-1}|_\sim<|C_{2,p}|_\sim=m$. Also, 
define the \emph{peak length (with respect to $C_2$)} of the factorisation  to  be the number of indices $j$ such that 
$|C_j|_\sim=m$; and assume the peak length of $\phi_1\cdots \phi_{n}$ is $M$. 
 From Lemma \ref{lem:Daypl}, there exist $\d_1,\ldots, \d_s\in \W_x$ such that $\phi_p\phi_{p+1}=\d_1\cdots \d_s$ in $P$ and 
$|C_{2,p-1}\d_1\cdots \d_t|_\sim<|C_{2,p}|_\sim$,  for $1\le t\le s-1$. 
 Therefore we have $\a=\phi_1\cdots \phi_{p-1}\d_1\cdots \d_s\phi_{p+1}\cdots \phi_n$ in $P$ and in this factorisation, either the 
 maximum height of peaks for $C_2$ is less than $m$, or the peak length is less than $M$.  This process may therefore be
 repeated until we obtain a factorisation of $\a$, in $P$, which is peak reduced with respect to $C_2$. 

Assume then that  $\a=\phi_1\cdots \phi_{n}$ in $P$, where $\phi_i\in \W_x$ and that this factorisation is peak reduced with
respect to $C_2$. From Lemma \ref{lem:nored} and the fact that $\a=1$ in $\St^\v_{x,l}$, 
we have $|C_{2,j}|_\sim=|C_2|_\sim$; so $\phi_j$ is of type $1_x$, $2a_x$ or $2b_x$, for $j=1,\ldots, n$. From Lemma \ref{lem:facta}
there is, in $P$,  a peak reduced factorisation $\a=\a_1\cdots \a_r\b_1\cdots \b_s\s$, where $\a_j$ is of Type $2a_x$, $\b_j$ is of Type $2b_x$ and
$\s$  is of Type $1_x$. Lemma \ref{lem:aone} then implies that $\a=1$ in $P$, as required.
\end{proof}
%
%
\bibliographystyle{plain}
\bibliography{aut}
\end{document}